\documentclass[11pt]{ip-journal-jdg}

\usepackage{mathrsfs}
\usepackage{amsfonts}
\usepackage{}

\usepackage{amsmath}
\usepackage{amscd}
\usepackage{amssymb}

\newcommand{\cal}{\mathcal}

\newcommand{\bC}{{\Bbb C}}
\newcommand{\bE}{{\Bbb E}}
\newcommand{\bL}{{\Bbb L}}

\newcommand{\bP}{{\Bbb P}}

\newcommand{\bZ}{{\Bbb Z}}

\newcommand{\cB}{{\cal B}}

\newcommand{\cD}{{\cal D}}

\newcommand{\cM}{{\cal M}}
\newcommand{\cO}{{\cal O}}

\newcommand{\cQ}{{\cal Q}}
\newcommand{\cT}{{\cal T}}

\newcommand{\cU}{{\cal U}}

\newcommand{\Mbar}{\overline{\cM}}

\DeclareMathOperator{\Gcd}{gcd}
\DeclareMathOperator{\Aut}{Aut}
\DeclareMathOperator{\Ext}{Ext}

\DeclareMathOperator{\val}{val}

\DeclareMathOperator{\Br}{Br}

\DeclareMathOperator{\rk}{rk}

\newcommand{\vir}{ {\mathrm{vir}} }

\newcommand{\tpi  }{\tilde{\pi}  }

\newcommand{\tB}{\tilde{B}}

\newcommand{\Mb}{\Mbar^\bullet_{\chi, \gamma}(\bP^1_{a}, \mu)}
\newcommand{\MP}{\Mbar^\bullet_\chi(\bP^1,\nu,\mu)}
\newcommand{\MQ}{\Mbar^\bullet_{\chi^1}(\bP^1,\nu,\mu)}
\newcommand{\Md}{\Mbar^\bullet_{\chi^0, \gamma-\nu}(\cB\bZ_{a})}

\newcommand{\xm}[1]{{#1}^\bullet_{\chi,\nu,\mu}}
\newcommand{\xn}[1]{{#1}^\bullet_{\chi,\mu,\gamma}}
\newcommand{\xo}[1]{{#1}^\bullet_{\mu}}

\newcommand{\amm}{|\Aut(\nu)||\Aut(\mu)|}
\newcommand{\am}{|\Aut(\mu)|}

\newtheorem{theorem}{Theorem}[section]
\newtheorem{Theorem}{Theorem}
\newtheorem{proposition}[theorem]{Proposition}
\newtheorem{lemma}[theorem]{Lemma}

\theoremstyle{remark}

\theoremstyle{definition}

\begin{document}
\title[Generalized Mari\~{n}o-Vafa Formula]{Generalized Mari\~{n}o-Vafa Formula and Local Gromov-Witten Theory of Orbi-curves}
\author{Zhengyu Zong}
\address{Department of Mathematics, Columbia University, New York, NY 10027, USA}
\email{zz2197@math.columbia.edu}
\begin{abstract}
We prove a generalized Mari\~{n}o-Vafa formula for Hodge integrals over $\Mbar_{g, \gamma-\mu}(\cB G)$ with $G$ an arbitrary finite abelian group. This formula can be viewed as a formula for the one-leg orbifold Gromov-Witten vertex where the leg is effective. We will prove the orbifold Gromov-Witten/Donaldson-Thomas correspondence between our formula and the formula for the orbifold DT vertex in \cite{Bry-Cad-You}. We will also use this formula to study the local Gromov-Witten theory of an orbi-curve with cyclic stack points in a Calabi-Yau three-orbifold.
\end{abstract}
\maketitle
\section{Introduction}\label{introduction}

The Gromov-Witten/Donaldson-Thomas correspondence conjectured in \cite{MNOP1,MNOP2} states that the GW theory and the DT theory of a smooth 3-fold are equivalent after a change of variables. By the results in \cite{MNOP1,Oko-Vaf}, this correspondence for smooth toric 3-folds is equivalent to the algorithm of the topological vertex \cite{AKMV}. The GW/DT correspondence for smooth toric 3-folds is proved in \cite{MOOP}. For smooth toric Calabi-Yau 3-folds, the GW theory is obtained by gluing the GW vertex, a generating function of cubic Hodge integrals, and the DT theory is obtained by gluing the DT vertex, a generating function of 3d partitions. The GW/DT correspondence for smooth toric Calabi-Yau 3-folds can be reduced to the correspondence between the GW vertex and the DT vertex. The Mari\~{n}o-Vafa formula, conjectured in \cite{Mar-Vaf} and proved in \cite{Liu-Liu-Zhou1,Oko-Pan}, can be viewed as a formula for the framed 1-leg GW vertex; it implies the correspondence between the 1-leg GW vertex and the 1-leg DT vertex.

A vertex formalism for the orbifold DT theory (resp. orbifold GW theory) of toric Calabi-Yau 3-orbifolds is established in \cite{Bry-Cad-You} (resp. \cite{Ros}). For toric Calabi-Yau 3-orbifolds, the orbifold GW theory is obtained by gluing the GW orbifold vertex, a generating function of cubic abelian Hurwitz-Hodge integrals, and the orbifold DT theory is obtained by gluing the DT orbifold vertex, a generating function of colored 3d partitions. J. Bryan showed the orbifold GW theory of the local footballs (computed in \cite{Joh-Pan-Tse}) and the orbifold DT theory of the local footballs (computed in \cite{Bry-Cad-You}) are equivalent after a change of variables \cite{Bry}. It is natural to ask if the GW orbifold vertex and the DT orbifold vertex are equivalent after a change of variables.

In this paper, we prove a generalized Mari\~{n}o-Vafa formula, which can be viewed as a formula for the framed 1-leg GW orbifold vertex where the leg is effective, and prove a GW/DT correspondence in this case. The correspondence for the 1-leg $\bZ_2$ vertex is proved in \cite{Ros}. We also use our generalized Mari\~{n}o-Vafa formula to study the local Gromov-Witten theory of an orbi-curve with cyclic stack points in a Calabi-Yau three-orbifold.

\subsection{ Mari\~{n}o-Vafa formula for $\bZ_a$}
Fix an integer $a\geq 1$. Let $\Mbar_{g, \gamma}(\cB\bZ_a)$ be the moduli space of stable maps to $\cB\bZ_a$ where $\gamma=(\gamma_{1}, \cdots, \gamma_{n})$ is a vector of elements in $\bZ_a$. Let $U$ be the irreducible representation of $\bZ_a$ given by
$$\phi^U:\bZ_a\to \bC^*,\phi^U(1)=e^{\frac{2\pi\sqrt{-1}}{a}}.$$
Then there is a corresponding Hodge bundle
$$E^U\to\Mbar_{g, \gamma}(\cB\bZ_a)$$
and the corresponding Hodge classes on $\Mbar_{g, \gamma}(\cB\bZ_a)$ are defined by Chern classes of $E^U$,
$$\lambda^U_i=c_i(E^U).$$
Similarly, for any irreducible representation $R$ of $\bZ_a$, we have a corresponding Hodge bundle $E^R$ and Hodge classes $\lambda^R_i$. Let $\Mbar_{g,n}$ be the moduli space of stable curves of genus $g$ with $n$ marked points and let $\psi_i$ be the $i^{th}$ descendent class on $\Mbar_{g,n}$, $1\leq i\leq n$. Let
$$\epsilon :\Mbar_{g, \gamma}(\cB\bZ_a)\to \Mbar_{g,n}$$
be the canonical morphism. Then the descendent classes $\bar{\psi}_i$ on $\Mbar_{g, \gamma}(\cB\bZ_a)$ are defined by
$$\bar{\psi}_i=\epsilon^*(\psi_i).$$
Let
$$
\Lambda^{\vee,R}_g(u)=u^{\rk E^R}-\lambda_1^R u^{\rk E^R-1} +\cdots+(-1)^{\rk E^R}\lambda_{\rk E^R}^R,
$$
where $\rk E^R$ is the rank of $E^R$ determined by the orbifold Riemann-Roch formula.

Let $d$ be a positive integer and let $\mu=(\mu_{1}\geq\cdots\geq\mu_{l(\mu)}>0)$ be a partition of $d>0$ which means $|\mu|:=\sum_{i=1}^{l(\mu)}\mu_i=d$. Now we require $\gamma=(\gamma_{1}, \cdots, \gamma_{n})$ to be a vector of \emph{nontrivial} elements in $\bZ_a$ and view $\mu$ as a vector of elements in $\bZ_a$. Then for $\tau\in \bZ$, we define $G_{g,\mu,\gamma}(\tau)_{a}$ as
\begin{eqnarray*}
&&\frac{\sqrt{-1}^{l(\mu)-|\mu|+2\sum_{i=1}^{l(\mu)}[\frac{\mu_i}{a}]}
\tau ^{l(\mu)-1}a^{l(\mu)-\sum_{i=1}^{l(\mu)}\delta_{0,\langle \frac{\mu_i}{a}\rangle}}}{\am|\Aut(\gamma)|}
\prod_{i=1}^{l(\mu)}\frac{\prod_{l=1}^{[\frac{\mu_i}{a}]}(\mu_i\tau+l)}{[\frac{\mu_i}{a}]!}
\\
&&\cdot\int_{\Mbar_{g, \gamma-\mu}(\cB\bZ_a)}\frac{\left(-\frac{1}{a}(\tau+\frac{1}{a})\right)^{-\delta}
\Lambda_{g}^{\vee,U}(\frac{1}{a})\Lambda_{g}^{\vee,U^\vee}(-\tau-\frac{1}{a})
\Lambda_{g}^{\vee,1}(\tau)}{\prod_{i=1}^{l(\mu)}(1-\mu_i\bar{\psi}_i)},
\end{eqnarray*}
where $\gamma-\mu$ denotes the vector $(\gamma_{1}, \cdots, \gamma_{n},-\mu_1,\cdots,-\mu_{l(\mu)})$, $\bar{\psi}_i$ corresponds to $\mu_i$, $U^\vee$ and 1 denote the dual of $U$ and the trivial representation respectively, $[x]$ denotes the integer part of $x$, $\langle x\rangle=x-[x]$,
$$\delta_{0,x}=\left\{\begin{array}{ll}1, &x=0,\\
0, &x\neq 0,\end{array} \right.$$
and
$$\delta=\left\{\begin{array}{ll}1, &\textrm{if all monodromies around loops on the domain curve are trivial,}\\
0, &\textrm{otherwise}.\end{array} \right.$$
Introduce formal variables $p=(p_1,p_2,\ldots,p_n,\ldots),x=(x_1,\ldots,x_{a-1})$ and define
$$
p_\mu=p_{\mu_1}\cdots p_{\mu_{l(\mu)} },x_\gamma=x_{\gamma_1}\cdots x_{\gamma_n}
$$
for a partition $\mu$. Define
generating functions
\begin{eqnarray*}
G_{\mu,\gamma}(\lambda;\tau)_{a}&=&\sum_{g=0}^{\infty}\lambda^{2g-2+l(\mu)}G_{g,\mu,\gamma}(\tau)_{a}\\
G(\lambda;\tau;p;x)_{a}&=&\sum_{\mu\neq \emptyset,\gamma}G_{\mu,\gamma}(\lambda;\tau)_{a}p_\mu x_\gamma=\sum_{\mu\neq \emptyset}G_{\mu}(\lambda;\tau;x)_{a}p_\mu=\sum_{\gamma}G_{\gamma}(\lambda;\tau;p)_{a}x_\gamma\\
G^\bullet(\lambda;\tau;p;x)_{a}&=&\exp(G(\lambda;\tau;p;x)_{a})=
\sum_{\mu,\gamma}G^\bullet_{\mu,\gamma}(\lambda;\tau)_{a}p_\mu x_\gamma=
1+\sum_{\mu\neq\emptyset,\gamma}G^\bullet_{\mu,\gamma}(\lambda;\tau)_{a}p_\mu x_\gamma\\
&=&1+\sum_{\mu\neq\emptyset}G^\bullet_{\mu}(\lambda;\tau;x)_{a}p_\mu\\
G^\bullet_{\mu,\gamma}(\lambda;\tau)_{a}&=&\sum_{\chi\in 2\bZ,\chi\leq 2l(\mu)}\lambda^{-\chi+l(\mu)}G^\bullet_{\chi,\mu,\gamma}(\tau)_{a}.
\end{eqnarray*}

Our $G(\lambda;\tau;p;x)_{a}$ corresponds to the framed 1-leg orbifold Gromov-Witten vertex where the leg is effective. In this paper, we will prove the orbifold GW/DT correspondence between our GW vertex $G(\lambda;0;p;x)_a$ and the 1-leg DT vertex $V_{\nu\emptyset\emptyset}^a(q,q_1,\cdots,q_{a-1})$ which appears in Example 4.2 in \cite{Bry-Cad-You}.

More precisely, let $V_{\nu\emptyset\emptyset}^{'a}(q,q_1,\cdots,q_{a-1})=
\frac{V_{\nu\emptyset\emptyset}^a(q,q_1,\cdots,q_{a-1})}{V_{\emptyset\emptyset\emptyset}^a(q,q_1,\cdots,q_{a-1})}$ be the corresponding reduced DT vertex. Then we have the following theorem:\\
\\
\textbf{Theorem 2.1}: Under the change of variables $q=-e^{\sqrt{-1}\lambda},q_l=\xi_a^{-1}e^{-\sum_{i=1}^{a-1}\frac{\omega_a^{-2il}}{a}(\omega_a^i-\omega_a^{-i})x_i},l=1,
\cdots,a-1$ we have
\begin{eqnarray*}
G^\bullet_{\mu}(\lambda;0;x)_{a}&=&\sum_{|\nu|=|\mu|}(-1)^{A_\nu(0,a)}q^{\frac{|\mu|}{2}+A_\nu(0,a)}q_1^{-\frac{d}{a}
+A_\nu(1,a)-A_\nu(0,a)}\cdots q_{a-1}^{-\frac{d(a-1)}{a}+A_\nu(a-1,a)-A_\nu(0,a)}\\
&&V_{\nu\emptyset\emptyset}^{'a}(-q,q_1,\cdots,q_{a-1})
\frac{\chi_{\nu}(\mu)}{z_\mu},
\end{eqnarray*}
where $\xi_a=e^{\frac{2\pi i}{a}},\omega_a=e^{\frac{\pi i}{a}}$ and $z_\mu=|\Aut(\mu)|\mu_1\cdots\mu_{l(\mu)}$.

For a partition $\mu$, $\chi_{\mu}$ denotes the character of the irreducible representation of $S_d$
indexed by $\mu$, where $d=|\mu| = \sum_{i=1}^{l(\mu)} \mu_i$. Let
$$\kappa_{\mu} = |\mu| + \sum_i (\mu_i^2 - 2i\mu_i).$$
Then under the above change of variables, let
\begin{eqnarray*}
R^\bullet_{\mu}(\lambda;\tau;x)_{a}&=&\sum_{|\nu|=|\mu|}(q^{\frac{1}{2}}q_1^{-\frac{1}{a}}\cdots q_{a-1}^{-\frac{a-1}{a}})^{|\nu|}\sum_{|\xi|=|\nu|}s_{\xi'}(-q_{\bullet})\chi_{\xi}(\nu)
\Phi^\bullet_{\nu,\mu}(\sqrt{-1}\lambda\tau)
\end{eqnarray*}
where
\begin{eqnarray*}
\Phi^\bullet_{\nu,\mu}(\lambda)=\sum_{\eta}\frac{\chi_{\eta}(\nu)}{z_{\nu}} \frac{\chi_{\eta}(\mu)}{z_{\mu}}e^{\frac{\kappa_{\eta}\lambda}{2}},
\end{eqnarray*}
$s_{\xi'}$ is the Schur polynomial corresponding to the conjugate of $\xi$, $-q_{\bullet}=(-Q,-Qq_{a-1},\cdots, -Qq_1\cdots q_{a-1})$, and $-Q=(1,-q,(-q)^2,(-q)^3,\cdots)$.

$\Phi^\bullet_{\nu,\mu}(\lambda)$ satisfies the following properties (see \cite{Liu-Liu-Zhou2}),

\begin{equation}\label{eqn:Comp}
  \Phi^{\bullet}_{\nu, \mu}(\lambda_1+\lambda_2)
=  \sum_{\sigma} \Phi^{\bullet}_{\nu, \sigma}(\lambda_1)
\cdot z_{\sigma} \cdot \Phi^{\bullet}_{\sigma, \mu}(\lambda_2),
\end{equation}
\begin{equation}\label{eqn:Init}
 \Phi^{\bullet}_{\nu, \mu}(0) = \frac{1}{z_{\nu}}\delta_{\nu, \mu}.
\end{equation}
By (\ref{eqn:Comp}), (\ref{eqn:Init}) we have
\begin{eqnarray*}
R^\bullet_{\mu}(\lambda;0;x)_{a}&=&(q^{\frac{1}{2}}q_1^{-\frac{1}{a}}\cdots q_{a-1}^{-\frac{a-1}{a}})^{|\mu|}\sum_{|\nu|=|\mu|}s_{\nu'}(-q_{\bullet})\frac{\chi_{\nu}(\mu)}{z_\mu},
\end{eqnarray*}
\begin{equation}\label{eqn:R}
R^\bullet_{\mu}(\lambda;\tau;x)_{a}=\sum_{|\nu|=|\mu|}R^\bullet_{\nu}(\lambda;0;x)_{a}z_{\nu}
\Phi^\bullet_{\nu,\mu}(\sqrt{-1}\lambda\tau).
\end{equation}

We also define generating functions

\begin{eqnarray*}
R^\bullet(\lambda;\tau;p;x)_{a}&=&\sum_{\mu,\gamma}R^\bullet_{\mu}(\lambda;\tau;x)_{a}p_\mu,\\
R(\lambda;\tau;p;x)_{a}&=&\log R^\bullet(\lambda;\tau;p;x)_{a}=\sum_{\mu\neq\emptyset,\gamma}R_{\mu,\gamma}(\lambda;\tau)_{a}p_\mu x_\gamma\\
&=&\sum_{\mu\neq \emptyset}R_{\mu}(\lambda;\tau;x)_{a}p_\mu=\sum_{\gamma}R_{\gamma}(\lambda;\tau;p)_{a}x_\gamma.
\end{eqnarray*}
\\

The Mari\~{n}o-Vafa formula for $\bZ_a$ is the following theorem.
\\
\begin{Theorem}[Mari\~{n}o-Vafa formula for $\bZ_a$]\label{MVza}
\begin{eqnarray*}
G(\lambda;\tau;p;q)_{a}=R(\lambda;\tau;p;x)_{a}.
\end{eqnarray*}
\end{Theorem}

When $a=1$ (and hence $x=\emptyset$),
\begin{eqnarray*}
R^\bullet_{\mu}(\lambda;0;\emptyset)_{a}&=&
(q^{\frac{1}{2}})^{|\mu|}\sum_{|\nu|=|\mu|}s_{\nu'}((1,-q,(-q)^2,(-q)^3,\cdots)\frac{\chi_{\nu}(\mu)}{z_\mu}\\
&=&\sum_{|\nu|=|\mu|} \frac{\chi_{\nu}(\mu)}{z_{\mu}}e^{\frac{\sqrt{-1}}{4}\kappa_{\nu}\lambda} V_{\nu}(\lambda)
\end{eqnarray*}
where
\begin{eqnarray*}
V_{\nu}(\lambda) &= & \prod_{1 \leq a < b \leq l(\nu)}
\frac{\sin \left[(\nu_a - \nu_b + b - a)\lambda/2\right]}{\sin \left[(b-a)\lambda/2\right]} \\
&& \cdot\frac{1}{\prod_{i=1}^{l(\nu)}\prod_{v=1}^{\nu_i} 2 \sin \left[(v-i+l(\nu))\lambda/2\right]}.
\end{eqnarray*}
$V_{\nu}(\lambda)$ has an interpretation in terms of quantum dimensions (see \cite{Liu-Liu-Zhou1} and \cite{Oko-Pan})
$$V_{\nu}(\lambda)=\sqrt{-1}^{|\nu|}\frac{\textrm{dim}_{q}R_{\nu}}{|\nu|!}=
\frac{\sqrt{-1}^{|\nu|}}{\prod_{x\in\nu}(q^{h(x)/2}-q^{-h(x)/2})},$$
where $q=e^{\sqrt{-1}\lambda}$, $\textrm{dim}_{q}R_{\nu}$ is the quantum dimension of the irreducible representation of $S_d$ index by $\nu$, the product is over all the boxes $x$ in the Young diagram associated to the partition $\nu$, and $h(x)$ is the hook length of the box $x$. In this case, Theorem \ref{MVza} is the original Mari\~{n}o-Vafa formula in \cite{Liu-Liu-Zhou1} and \cite{Oko-Pan}. In \cite{Liu}, $G(\lambda;\tau;p;\emptyset)_{1}$ and $R(\lambda;\tau;p;\emptyset)_{1}$ are denoted by $G(\lambda;\tau;p)$ and $R(\lambda;\tau;p)$ respectively.

In section 3, we will define a generating function $\xo{K}(\lambda;x)$ of relative Gromov-Witten invariants of $(\bP^1_a,\infty)$. In section 4, we will prove the following theorem, which gives Theorem \ref{MVza}.

\begin{Theorem}\label{key}
\begin{eqnarray*}
\xo{K}(\lambda;x)=\sum_{|\nu|=|\mu|}G^\bullet_{\nu}(\lambda;\tau;x)_{a}z_{\nu}
\Phi^\bullet_{\nu,\mu}(-\sqrt{-1}\tau\lambda).
\end{eqnarray*}
\end{Theorem}

By (\ref{eqn:Comp}), (\ref{eqn:Init}), Theorem \ref{key} is equivalent to
$$G^\bullet_{\mu}(\lambda;\tau;x)_{a}=\sum_{|\nu|=|\mu|}K^\bullet_{\nu}(\lambda;x)z_{\nu}
\Phi^\bullet_{\nu,\mu}(\sqrt{-1}\tau\lambda),\ \ \ G^\bullet_{\mu}(\lambda;0;x)_{a}=\xo{K}(\lambda;x).$$
So Theorem \ref{MVza} follows from (\ref{eqn:R}) and the initial condition $G^\bullet_{\mu}(\lambda;0;x)_{a}=R^\bullet_{\mu}(\lambda;0;x)_{a}$.

\subsection{Abelian group $G$}

Let $G$ be an finite abelian group and let $R$ be an irreducible representation
$$\phi^{R}:G\to \bC^*$$
with associated short exact sequence
$$0\to K\to G\stackrel{\phi^{R}}{\to}\textrm{Im}(\phi^{R})\cong \bZ_a\to 0.$$
Let $\bar{\mu}=(\bar{\mu}_{1},\cdots,\bar{\mu}_{l(\bar{\mu})})$ be a vector of elements in $G$ such that $$\phi^{R}(\bar{\mu})=\mu.$$
Here we view $\mu$ as a vector of elements in $\bZ_a$. Let $\gamma=(\gamma_{1}, \cdots, \gamma_{n})$ be a vector of elements in $G$ such that $\phi^{R}(\gamma)$ is a vector of \emph{nontrivial} elements in $\bZ_a$. Then we define $G_{g,\mu,\gamma}(\tau)_{a,\phi^R}$ as
\begin{eqnarray*}
&&\frac{\sqrt{-1}^{l(\mu)-|\mu|+2\sum_{i=1}^{l(\mu)}[\frac{\mu_i}{a}]}
\tau ^{l(\mu)-1}a^{l(\mu)-\sum_{i=1}^{l(\mu)}\delta_{0,\langle \frac{\mu_i}{a}\rangle}}}{\am|\Aut(\gamma)|}
\prod_{i=1}^{l(\mu)}\frac{\prod_{l=1}^{[\frac{\mu_i}{a}]}(\mu_i\tau+l)}{[\frac{\mu_i}{a}]!}\\
&&\cdot\int_{\Mbar_{g, \gamma-\bar{\mu}}(\cB G)}\frac{\left(-\frac{1}{a}(\tau+\frac{1}{a})\right)^{-\delta_K}
\Lambda_{g}^{\vee,R}(\frac{1}{a})\Lambda_{g}^{\vee,R^\vee}(-\tau-\frac{1}{a})
\Lambda_{g}^{\vee,1}(\tau )}{\prod_{i=1}^{l(\mu)}(1-\mu_{i}\bar{\psi}_{i})}
\end{eqnarray*}
where
$$\delta_K=\left\{\begin{array}{ll}1, &\textrm{if all monodromies around loops on the domain curve are contained in $K$},\\
0, &\textrm{otherwise}.\end{array} \right.$$
In the definition of $G_{g,\mu,\gamma}(\tau)_{a,\phi^R}$, we have some freedoms of choosing $\bar{\mu}$ and we set the following requirements:
\begin{enumerate}
\item If $\sum_{i=1}^{n}\phi^{R}(\gamma_i)-\sum_{j=1}^{l(\mu)}\mu_j=0$, then $\sum_{i=1}^{n}\gamma_i-\sum_{j=1}^{l(\mu)}\bar{\mu}_j\in K$. So we can choose $\bar{\mu}$ such that $\phi^{R}(\bar{\mu})=\mu$ still holds and $\sum_{i=1}^{n}\gamma_i-\sum_{j=1}^{l(\mu)}\bar{\mu}_j=0$.
\item If $\sum_{i=1}^{n}\phi^{R}(\gamma_i)-\sum_{j=1}^{l(\mu)}\mu_j\neq 0$, then both $\Mbar_{g, \gamma-\bar{\mu}}(\cB G)$ and $\Mbar_{g, \phi^{R}(\gamma)-\mu}(\cB\bZ_a)$ are empty and we choose $\bar{\mu}$ arbitrarily.
\end{enumerate}
We also define generating functions
\begin{eqnarray*}
G_{\mu,\gamma}(\lambda;\tau)_{a,\phi^R}&=&\sum_{g=0}^{\infty}\lambda^{2g-2+l(\mu)}G_{g,\mu,\gamma}(\tau)_{a,\phi^R}\\
G_\gamma(\lambda;\tau;p)_{a,\phi^R}&=&\sum_{\mu\neq \emptyset}G_{\mu,\gamma}(\lambda;\tau)_{a,\phi^R}p_\mu.\\
\end{eqnarray*}

The Mari\~{n}o-Vafa formula for $G$ is the following theorem.
\\
\begin{Theorem}[Mari\~{n}o-Vafa formula for $G$]\label{MVG}
\begin{eqnarray*}
G_\gamma(\lambda;\tau;p)_{a,\phi^R}=|K|R_{\phi^R(\gamma)}(|K|\lambda;\tau;\frac{p}{|K|})_{a},
\end{eqnarray*}
where $\frac{p}{|K|}$ denotes the formal variables $(\frac{p_1}{|K|},\frac{p_2}{|K|},\ldots,\frac{p_n}{|K|},\ldots)$.
\end{Theorem}

\subsection{The local Gromov-Witten theory of orbi-curves}
Let $(X,p_1,\cdots,p_r,q_1,\cdots,q_s)$ be a fixed non-singular genus $g$ orbi-curve with stack points $p_1,\cdots,p_r$ of orders $a_1,\cdots,a_r$ and with ordinary points $q_1,\cdots,q_s$. Let $\alpha^1,\cdots,\alpha^s$ be partitions of $d>0$. Let $\gamma^1,\cdots,\gamma^r$ be vectors of nontrivial elements in $\bZ_{a_1},\cdots,\bZ_{a_r}$ respectively. When $a_1=\cdots=a_r=1$, the \emph{relative local invariant} $Z^b_d(g)_{\overrightarrow{\alpha}}$ of $X$ is defined in \cite{Bry-Pan1}. This local theory is called \emph{local Calabi-Yau theory} in \cite{Bry-Pan2} because the vector bundle over the target curve in the definition of $Z^b_d(g)_{\overrightarrow{\alpha}}$ is a (non-compact) Calabi-Yau threefold. In \cite{Bry-Pan2}, the local Gromov-Witten theory of ordinary curves is solved without imposing the Calabi-Yau condition on the target.

We will define the \emph{relative local invariant} $Z^{b,\gamma}_d(g)_{\overrightarrow{\alpha},\overrightarrow{a}}$ of $(X,p_1,\cdots,p_r,q_1,\cdots,q_s)$ similar to the one in \cite{Bry-Pan1}, where $\overrightarrow{a}=(a_1,\cdots,a_r)$ and $\gamma=\gamma^1+\cdots+\gamma^r$. Then to determine the relative local invariants of all $(X,p_1,\cdots,p_r,q_1,\cdots,q_s)$, we only need to determine $Z^{b,\gamma}_d(0)_{(\mu),(a)}$ because of the gluing law in \cite{Bry-Pan1}. Define the generating function
$$Z_d(g)(\lambda;x)_{\overrightarrow{\alpha},\overrightarrow{a}}=
\sum_{b,\gamma}Z^{b,\gamma}_d(g)_{\overrightarrow{\alpha},\overrightarrow{a}}\lambda^bx_\gamma.$$
We will use the Mari\~{n}o-Vafa formula for $\bZ_a$ to calculate $Z_d(0)(\lambda;x)_{(\mu),(a)}$. The result is the following theorem:
\begin{Theorem}
\begin{eqnarray*}
\lambda^{-\frac{d}{a}}Z_d(0)(\lambda;\lambda^{\frac{1}{a}-1}x_1,\cdots,\lambda^{\frac{a-1}{a}-1}x_{a-1})
_{(\mu),(a)}=\sqrt{-1}^{d-l(\mu)}(q^{\frac{1}{2}}q_1^{-\frac{1}{a}}\cdots q_{a-1}^{-\frac{a-1}{a}})^{|\mu|}\sum_{|\nu|=|\mu|}s_{\nu'}(-q_{\bullet})\frac{\chi_{\nu}(\mu)}{z_\mu},
\end{eqnarray*}
where the change of variables is given in Theorem \ref{correspondence}.
\end{Theorem}

\subsection{Acknowledgments}

I wish to express my deepest thanks to my advisor Chiu-Chu Melissa Liu. When I started this work, she helped me to find many important references, from which I learned the orbifold Gromov-Witten theory and the Mari\~{n}o-Vafa formula. I have learned a lot of things from discussions with her. Her papers \cite{Liu},  \cite{Liu-Liu-Zhou1}, \cite{Liu-Liu-Zhou2} guided me through the whole process of this work. She also helped me to learn some skills of LaTeX. This paper would not have been possible without her. I also wish to thank Hsian-Hua Tseng and Jim Bryan for their helpful communications which are important to the study of the local Gromov-Witten theory of orbi-curves.

\section{Initial value and orbifold Gromov-Witten/Donaldson-Thomas correspondence}\label{initial}
The original GW/DT correspondence is conjectured in \cite{MNOP1,MNOP2}, which states that the GW theory and the DT theory of a smooth 3-fold are equivalent after a change of variables. In \cite{MOOP}, this conjecture is proven for smooth toric 3-folds . This result can be viewed as the  equivalence of the GW theory and the DT theory of the non-orbifold topological vertex. In \cite{Bry-Cad-You}, the DT theory of the orbifold topological vertex is established. We will prove the orbifold GW/DT correspondence between our GW vertex $G(\lambda;0;p;x)_a$ and the one-leg DT vertex $V_{\nu\emptyset\emptyset}^a(q,q_1,\cdots,q_{a-1})$ which appears in Example 4.2 in \cite{Bry-Cad-You}. When $a=2$, this correspondence is proven in \cite{Ros}.

In \cite{Bry-Cad-You}, the DT vertex $V_{\nu\emptyset\emptyset}^a(q,q_1,\cdots,q_{a-1})$ is expressed as 
\begin{eqnarray*}
&&V_{\nu\emptyset\emptyset}^a(q,q_1,\cdots,q_{a-1})\\
&=&V_{\emptyset\emptyset\emptyset}^a(q,q_1,\cdots,q_{a-1})q^{-A_\nu(0,a)}
q_1^{A_\nu(0,a)-A_\nu(1,a)}\cdots q_{a-1}^{A_\nu(0,a)-A_\nu(a-1,a)}s_{\nu'}(q_{\bullet}).
\end{eqnarray*}
where $A_\nu(k,n)=\sum_{(i,j)\in\nu}[\frac{i+k}{n}]$, $s_\nu'$ is the Schur polynomial, $\nu'$ is the conjugate of $\nu$ and $q_{\bullet}=(Q,Qq_{a-1},\cdots, Qq_1\cdots q_{a-1}), Q=(1,q,q^2,q^3,\cdots)$. Let $V_{\nu\emptyset\emptyset}^{'a}(q,q_1,\cdots,q_{a-1})=
\frac{V_{\nu\emptyset\emptyset}^a(q,q_1,\cdots,q_{a-1})}{V_{\emptyset\emptyset\emptyset}^a(q,q_1,\cdots,q_{a-1})}$ be the corresponding reduced DT vertex. The following theorem gives the orbifold GW/DT correspondence between $G^\bullet_{\mu}(\lambda;0;x)_{a}$ and $V_{\nu\emptyset\emptyset}^{'a}(q,q_1,\cdots,q_{a-1})$.
\begin{theorem}\label{correspondence}
Under the change of variables $q=-e^{\sqrt{-1}\lambda},q_l=\xi_a^{-1}e^{-\sum_{i=1}^{a-1}\frac{\omega_a^{-2il}}{a}(\omega_a^i-\omega_a^{-i})x_i},l=1,
\cdots,a-1$ we have
\begin{eqnarray*}
G^\bullet_{\mu}(\lambda;0;x)_{a}&=&\sum_{|\nu|=|\mu|}(-1)^{A_\nu(0,a)}q^{\frac{|\mu|}{2}+A_\nu(0,a)}q_1^{-\frac{d}{a}
+A_\nu(1,a)-A_\nu(0,a)}\cdots q_{a-1}^{-\frac{d(a-1)}{a}+A_\nu(a-1,a)-A_\nu(0,a)}\\
&&V_{\nu\emptyset\emptyset}^{'a}(-q,q_1,\cdots,q_{a-1})
\frac{\chi_{\nu}(\mu)}{z_\mu},
\end{eqnarray*}
where $\xi_a=e^{\frac{2\pi i}{a}},\omega_a=e^{\frac{\pi i}{a}}$ and $z_\mu=|\Aut(\mu)|\mu_1\cdots\mu_{l(\mu)}$.
\end{theorem}
\begin{proof}
If $\tau=0$, then $G_{g,\mu,\gamma}(0)_{a}=0$ if $l(\mu)>1$. If $l(\mu)=1$, then the  moduli space $\Mbar_{g, \gamma-(d)}(\cB\bZ_{a})$ is nonempty if and only if the parity condition
$$ d=\sum_{i=1}^{n}\gamma_{i}\textrm{ (mod $a$)}$$
holds. So when $l(\mu)=1$ and $ d=\sum_{i=1}^{n}\gamma_{i}\textrm{ (mod $a$)}$, we have
\begin{eqnarray*}
G_{g,(d),\gamma}(0)_{a}&=&\sqrt{-1}^{1-d+2[\frac{d}{a}]}a^{1-\delta_{0,\langle \frac{d}{a}\rangle}}\\
&&\frac{1}{|\Aut(\gamma)|}\int_{\Mbar_{g, \gamma-(d)}(\cB\bZ_{a})}\frac{\left(-\frac{1}{a^2}\right)^{-\delta}\Lambda_{g}^{\vee,U}(\frac{1}{a})\Lambda_{g}^{\vee,U^\vee}(-\frac{1}{a})
\Lambda_{g}^{\vee,1}(0)}{1-d\bar{\psi}}\\
&=&\sqrt{-1}^{1-d+2[\frac{d}{a}]}a^{1-\delta_{0,\langle \frac{d}{a}\rangle}}\\
&&\frac{1}{|\Aut(\gamma)|}\int_{\Mbar_{g, \gamma-(d)}(\cB\bZ_{a})}\frac{\frac{(-1)^{g-1+n-\sum_{i=1}^{n}\frac{\gamma_{i}}{a}+\langle \frac{d}{a}\rangle}}{a^{2g-1+n-\delta_{0,\langle \frac{d}{a}\rangle}}}(-1)^{g}\lambda_g}{1-d\bar{\psi}}\\
&=&\frac{(-1)^{n-1}\sqrt{-1}^{1-d+2(\frac{d}{a}-\sum_{i=1}^{n}\frac{\gamma_{i}}{a})}}{a^{2g-2+n}}\\
&&\frac{1}{|\Aut(\gamma)|}\int_{\Mbar_{g, \gamma-(d)}(\cB\bZ_{a})}\frac{\lambda_g}{1-d\bar{\psi}}
\end{eqnarray*}
\begin{eqnarray*}
&=&\frac{(-1)^{n-1}\sqrt{-1}^{1-d+2(\frac{d}{a}-\sum_{i=1}^{n}\frac{\gamma_{i}}{a})}}{a^{-1+n}}d^{2g-2+n}\\
&&\frac{1}{|\Aut(\gamma)|}\int_{\Mbar_{g, n+1}}\lambda_g\psi^{2g-2+n}\\
&=&\frac{(-1)^{n-1}\sqrt{-1}^{1-d+2(\frac{d}{a}-\sum_{i=1}^{n}\frac{\gamma_{i}}{a})}}{a^{-1+n}}d^{2g-2+n}\\
&&\frac{1}{|\Aut(\gamma)|}\int_{\Mbar_{g, 1}}\lambda_g\psi^{2g-2},
\end{eqnarray*}
where the second equality holds by Mumford's relation \cite{Bry-Gra-Pan}, the fourth equality holds by the fact that the degree of the map $\Mbar_{g, \gamma-(d)}(\cB\bZ_{a})\to \Mbar_{g, n+1}$ is $a^{2g-1}$ \cite{JPT}, and the last equality holds by the string equation.

So we have
\begin{eqnarray*}
G_{(d),\gamma}(\lambda;0)_{a}&=&\sum_{g=0}^{\infty}\lambda^{2g-1}G_{g,(d),\gamma}(0)_{a}\\
&=&\frac{(-1)^{n-1}\sqrt{-1}^{1-d+2(\frac{d}{a}-\sum_{i=1}^{n}\frac{\gamma_{i}}{a})}d^{n-1}}{2a^{-1+n}\sin \frac{d\lambda}{2}|\Aut(\gamma)|},
\end{eqnarray*}
where the second equality holds by the fact (see \cite{Fab-Pan})
$$\sum_{g=0}^{\infty}\lambda^{2g}\int_{\Mbar_{g, 1}}\lambda_g\psi^{2g-2}=\frac{\lambda/2}{\sin(\lambda/2)}.$$
So we have
\begin{eqnarray*}
G(\lambda;0;p;x)_{a}&=&\sum_{d\geq 1,d=\sum_{i=1}^{n}\gamma_{i}(\textrm{mod}a)}\frac{(-1)^{n-1}\sqrt{-1}^{1-d+2(\frac{d}{a}-\sum_{i=1}^{n}
\frac{\gamma_{i}}{a})}d^{n-1}}{2a^{-1+n}\sin \frac{d\lambda}{2}|\Aut(\gamma)|}p_d x_\gamma\\
&=&-\sum_{d\geq 1}\sum_{k_1,\cdots,k_{a-1}\geq 0}\sum_{l=0}^{a-1}\xi_{a}^{l(d-\sum_{j=1}^{a-1}jk_j)}\frac{\sqrt{-1}^{1-d+\frac{2d}{a}}p_d}{2d\sin \frac{d\lambda}{2}}\prod_{j=1}^{a-1}\frac{(-\frac{d}{a}\omega_{a}^{-j}x_j)^{k_j}}{k_j!}\\
&=&-\sum_{d\geq 1}\sum_{k_1,\cdots,k_{a-1}\geq 0}\sum_{l=0}^{a-1}\xi_{a}^{ld}\frac{\sqrt{-1}^{1-d+\frac{2d}{a}}p_d}{2d\sin \frac{d\lambda}{2}}\prod_{j=1}^{a-1}\frac{(-\frac{d}{a}\omega_{a}^{-j}\xi_a^{-jl}x_j)^{k_j}}{k_j!}.
\end{eqnarray*}
Let $u_l=\exp(\sum_{j=1}^{a-1}-\frac{1}{a}\omega_{a}^{-j}\xi_a^{-jl}x_j),l=0,\cdots,a-1$ and $p_d=\sum_{i=1}^{\infty}y_i^d$. Then we have
\begin{eqnarray*}
G(\lambda;0;p;x)_{a}&=&-\sum_{d\geq 1}\sum_{l=0}^{a-1}\frac{p_d}{d}\frac{(\sqrt{-1}^{-\frac{a-2}{a}}\xi_a^lu_l(-q)^{\frac{1}{2}})^d}{1-(-q)^{d}}\\
&=& -\sum_{d\geq 1}\sum_{l=0}^{a-1}\sum_{i,j\geq 1}\frac{1}{d}(\sqrt{-1}^{-\frac{a-2}{a}}\xi_a^lu_l(-q)^{\frac{1}{2}}y_i)^d(-q)^{d(j-1)}\\
&=&\log(\prod_{l=0}^{a-1}\prod_{i,j\geq 1}(1+\sqrt{-1}^{-\frac{a-2}{a}}\xi_a^lu_l(-q)^{\frac{1}{2}}y_i(-q)^jq^{-1})).
\end{eqnarray*}
Note that $\frac{u_{l-1}}{u_l}=\xi_aq_l$ for $l=1,\cdots,a-1$. So we have $u_l\xi_a^l=\frac{u_0}{q_1\cdots q_l},l=1,\cdots,a-1$. Also note that $u_0\cdots u_{a-1}=1$. So $u_l\xi_a^l=\frac{1}{u_1\cdots u_{a-1}q_1\cdots q_l}$. By taking the product of this identity for $l=1,\cdots,a-1$, we obtain
$$u_1\cdots u_{a-1}=(-1)^{\frac{a-1}{a}}\frac{1}{q_1^{\frac{a-1}{a}}q_2^{\frac{a-2}{a}}\cdots q_{a-1}^{\frac{1}{a}}}.$$
Thus we have
$$u_l\xi_a^l=(-1)^{\frac{a-1}{a}}q_1^{-\frac{1}{a}}\cdots q_{a-1}^{-\frac{a-1}{a}}q_{l+1}\cdots q_{a-1}.$$
Therefore
\begin{eqnarray*}
G(\lambda;0;p;x)_{a}&=&\log(\prod_{l=0}^{a-1}\prod_{i,j\geq 1}(1+q^{\frac{1}{2}}q_1^{-\frac{1}{a}}\cdots q_{a-1}^{-\frac{a-1}{a}}q_{l+1}\cdots q_{a-1}(-q)^{j-1}y_i))\\
&=&\log(\sum_{d\geq 1}(q^{\frac{1}{2}}q_1^{-\frac{1}{a}}\cdots q_{a-1}^{-\frac{a-1}{a}})^d\sum_{|\nu|=d}s_{\nu'}(-q_{\bullet})s_{\nu}(y))\\
&=&\log(\sum_{d\geq 1}(q^{\frac{1}{2}}q_1^{-\frac{1}{a}}\cdots q_{a-1}^{-\frac{a-1}{a}})^d\sum_{|\mu|=|\nu|=d}s_{\nu'}(-q_{\bullet})\frac{\chi_{\nu}(\mu)}{z_\mu}p_\mu(y)),
\end{eqnarray*}
where the second identity can be found in \cite{Mac},
$$s_{\nu}(y)=\sum_{|\mu|=|\nu|}\frac{\chi_{\nu}(\mu)}{z_\mu}p_\mu(y)$$
is the Schur polynomial, $-q_{\bullet}=(-Q,-Qq_{a-1},\cdots, -Qq_1\cdots q_{a-1}), -Q=(1,-q,(-q)^2,(-q)^3,\cdots)$, and $z_\mu=|\Aut(\mu)|\mu_1\cdots\mu_{l(\mu)}$.
Recall that $G^\bullet(\lambda;0;p;x)_{a}=\exp(G(\lambda;0;p;x)_{a})
=1+\sum_{\mu\neq\emptyset}G^\bullet_{\mu}(\lambda;0;x)_{a}p_\mu$.
So
\begin{eqnarray*}
G^\bullet_{\mu}(\lambda;0;x)_{a}&=&(q^{\frac{1}{2}}q_1^{-\frac{1}{a}}\cdots q_{a-1}^{-\frac{a-1}{a}})^{|\mu|}\sum_{|\nu|=|\mu|}s_{\nu'}(-q_{\bullet})\frac{\chi_{\nu}(\mu)}{z_\mu}\\
&=&\sum_{|\nu|=|\mu|}(-1)^{A_\nu(0,a)}q^{\frac{|\mu|}{2}+A_\nu(0,a)}q_1^{-\frac{d}{a}
+A_\nu(1,a)-A_\nu(0,a)}\cdots q_{a-1}^{-\frac{d(a-1)}{a}+A_\nu(a-1,a)-A_\nu(0,a)}\\
&&V_{\nu\emptyset\emptyset}^{'a}(-q,q_1,\cdots,q_{a-1})
\frac{\chi_{\nu}(\mu)}{z_\mu}.
\end{eqnarray*}
\end{proof}
Recall that
\begin{eqnarray*}
R^\bullet_{\mu}(\lambda;0;x)_{a}&=&(q^{\frac{1}{2}}q_1^{-\frac{1}{a}}\cdots q_{a-1}^{-\frac{a-1}{a}})^{|\mu|}\sum_{|\nu|=|\mu|}s_{\nu'}(-q_{\bullet})\frac{\chi_{\nu}(\mu)}{z_\mu}.
\end{eqnarray*}
Therefore, the initial condition $G^\bullet_{\mu}(\lambda;0;x)_{a}=R^\bullet_{\mu}(\lambda;0;x)_{a}$ follows from Theorem \ref{correspondence} and the expression of $V_{\nu\emptyset\emptyset}^{'a}$.
\section{Moduli spaces of relative stable morphisms} \label{moduli}
\subsection{Moduli spaces}
Fix an integer $a\geq 1$. Let $\bP^1_{a}$ be the projective line $\bP^1$ with root construction \cite{Cad} of order $a$ at 0. For an integer $m>0$, let
$$
\bP^1_{a}[m]=\bP^1_{a}\cup\bP^1_{(1)}\cup\cdots\cup\bP^1_{(m)}
$$
be the union of $\bP^1_{a}$ and a chain of $m$ copies $\bP^1$, where $\bP^1_{a}$ is glued to $\bP^1_{(1)}$ at  $p_1^{(0)}$ and $\bP^1_{(l)}$ is
glued to $\bP^1_{(l+1)}$ at $p_1^{(l)}$ for $1\leq l \leq m-1$. The distinguished point on $\bP^1_{(m)}$ is $p_1^{(m)}$. We call the irreducible component $\bP^1_{a}$ the root component and the other irreducible components the bubble components. Denote by $\pi[m]: \bP^1_{a}[m] \to \bP^1_{a}$ the map which
is identity on the root component and contracts all the bubble components
to $p_1^{(0)}$. Let
$$
\bP^1(m)=\bP^1_{(1)}\cup\cdots\cup\bP^1_{(m)}
$$
denote the union of bubble components of $\bP^1_{a}[m]$. For convenience, we set $\bP^1_{a}[0]=\bP^1_{(0)}=\bP^1_{a}$.

Let $\gamma=(\gamma_{1}, \cdots, \gamma_{n})$ be the vector of integers
$$1\leq \gamma_{i}\leq a-1$$
defining nontrivial elements $\gamma_{i}\in \bZ_{a}$. Let $\mu=(\mu_{1}\geq\cdots\geq\mu_{l(\mu)}>0)$ be a partition of $d>0$.
Let $\Mbar_{g, \gamma}(\bP^1_{a}, \mu)$ be the moduli space of relative maps to $(\bP^1_{a}, \infty)$. Then a point in $\Mbar_{g, \gamma}(\bP^1_{a}, \mu)$ is of the form
$$
[f: (C, x_1,\cdots , x_n, y_1, \cdots , y_{l(\mu)})\to (\bP^1_a,p_1^{(m)})]
$$
such that
$$
f^{-1}(p_1^{(m)})=\sum_{i=1}^{l(\mu)}\mu_iy_i
$$
as Cartier divisors.
In order for the moduli space $\Mbar_{g, \gamma}(\bP^1_{a}, \mu)$ to be non-empty, we must have the parity condition
$$ d=\sum_{i=1}^{n}\gamma_{i}\textrm{ (mod $a$)}.$$
We will also consider the disconnected version  $\Mbar^\bullet_{\chi, \gamma}(\bP^1_{a}, \mu)$, where the domain curve $C$ is allowed to be disconnected with $2(h^0(\cO_{c})-h^0(\cO_{c}))=\chi$.

Similarly, if we specify ramification types $\nu,\mu$ over $0,\infty\in\bP^1$, we can define the corresponding moduli spaces $\Mbar_{g, 0}(\bP^1,\nu, \mu)$ and $\Mbar^\bullet_{\chi}(\bP^1,\nu, \mu)$ of relative stable maps.

\subsection{Torus action}
Consider the $\bC^*$-action
$$t \cdot [z^0:z^1] = [tz^0: z^1]$$
on $\bP^1$. This action lifts canonically on $\bP^1_a$. These induce actions on $\bP^1[m]$ and on $\bP^1_a[m]$ with trivial actions on the bubble components. These in turn induce actions on $\Mbar_{g, \gamma}(\bP^1_{a}, \mu), \Mbar^\bullet_{\chi, \gamma}(\bP^1_{a}, \mu), \Mbar_{g, 0}(\bP^1,\nu, \mu)$, and $\Mbar^\bullet_{\chi}(\bP^1,\nu, \mu)$. Define the quotient space $\MP//\bC^*$ to be
$$\MP//\bC^*=(\MP\setminus \MP^{\bC^*})/\bC^*.$$

\subsection{The branch morphism and double Hurwitz numbers}
Similar to the case of $\Mbar_{g, \gamma}(\bP^1_{a}, \mu)$, a map $[f]\in \Mbar^\bullet_{\chi}(\bP^1,\nu, \mu)$ has target of the form $\bP^1[m_{0},m_1]$, where $\bP^1[m_{0},m_1]$ is obtained by attaching $\bP^1(m_0)$ and $\bP^1(m_1)$ to $\bP^1$ at 0 and $\infty$ respectively. The distinguished points on $\bP^1[m_{0},m_1]$ are $q^0_{m_0}$ and $q^1_{m_1}$. Let $\pi[m_{0},m_1]:\bP^1 [m_{0},m_1]\to \bP^1$ be the contraction to the root component. Let $r=-\chi+l(\nu)+l(\mu)$ be the virtual dimension of $\Mbar^\bullet_{\chi}(\bP^1,\nu, \mu)$. Then there is a branch morphism
$$ \textrm{Br}:\Mbar^\bullet_{\chi}(\bP^1,\nu, \mu)\to \textrm{Sym}^r\bP^1$$
sending $[f:C\to \bP^1[m_0,m_1]]$ to
$$\textrm{div} (\tilde{f})-(d-l(\nu))0-(d-l(\mu))\infty,$$
where div$(\tilde{f})$ is the branch divisor of $\tilde{f}=\pi [m_{0},m_1]\circ f:C\to \bP^1$.

Recall that the disconnected double Hurwitz number is defined by
$$
\xm{H} =\frac{1}{\amm}\int_{[\MP]^{\vir} }\Br^*( H^{r}).
$$
In \cite{Liu-Liu-Zhou2}, the following proposition is proved.
\\
\begin{proposition}
$$\xm{H}=\frac{r!}{\amm}\int_{[\MP//\bC^*]^{\vir} }(\psi^0)^{r-1}$$
where $\psi^0$ is the target $\psi$ class, the first Chern class of the line bundle $\bL_0$ over $\MP$ whose fiber at
$$[f:C\to \bP^1[m_{0},m_1]]$$
is the cotangent line $T^*_{q^0_{m_0}}\bP^1[m_{0},m_1]$.
\end{proposition}

The following Burnside formula will also be used (see \cite{Dij}).
\\
\begin{proposition}[Burnside formula]
$$\Phi^\bullet_{\nu,\mu}(\lambda)=\sum_{\chi}\lambda^{-\chi+l(\mu)+l(\nu)}\frac{\xm{H}}{(-\chi+l(\mu)+l(\nu))!}$$
where $\Phi^\bullet_{\nu,\mu}(\lambda)$ is defined combinatorially in section 1.
\end{proposition}

\subsection{The obstruction bundle}
Let
$$\pi:\cU\to \Mb$$
be the universal domain curve and let $\cT$ be the universal target. Then there is an evaluation map
$$F:\cU\to \cT$$
and a contraction map
$$\tpi: \cT\to \bP^1_a.$$
Let $\cD\subset \cU$ be the divisor corresponding to the $l(\mu)$ marked points $y_1,\cdots,y_{l(\mu)}$. Define
\begin{eqnarray*}
V_D&=&R^1\pi_*(\cO_{\cU}(-\mathcal{D}) )\\
V_{D_d}&=&R^1\pi_* \tilde{F}^*\cO_{\bP^1_a}(-p_{0}),
\end{eqnarray*}
where $\tilde{F}=\tilde\pi\circ F:\mathcal{U}\to \bP^1_a,p_1=\infty\in\bP^1_a$, and $p_{0}$ is the stack point of $\bP^1_a$. The fibers of $V_D$ and $V_{D_d}$ at
$$
\left[ f:(C,x_1,\ldots,x_n,y_1,\cdots,y_{l(\mu)})\to \bP^1_a[m]\ \right]\in \Mb
$$
are $H^1(C, \cO_C(-D))$ and $H^1(C, \tilde{f}^*\cO_{\bP^1_a}(-p_0))$,
respectively, where $D=y_1+\ldots+y_{l(\mu)}$ and $\tilde{f}=\pi[m]\circ f$. It is easy to see that the rank of the obstruction bundle
$$V=V_D\oplus V_{D_{d}}$$
is equal to the virtual dimension of $\Mb$, which is $-\chi+n+l(\mu)+\frac{d}{a}-\sum_{i=1}^{n}\frac{\gamma_i}{a}$.

We lift the $\bC^*$-action to the obstruction bundle $V$. It suffices to lift the $\bC^*$-action on $\bP^1_a$ to the line bundles $\cO_{\bP^1_a}(-p_{0})$ and $\cO_{\bP^1_a}$. Let the weights of the $\bC^*$-action on $\cO_{\bP^1_a}(-p_{0})$ at $p_0$ and $p_1$ be $-\tau-\frac{1}{a}$ and $-\tau$, respectively, and let the weights of the $\bC^*$-action on $\cO_{\bP^1_a}$ at $p_0$ and $p_1$ be $\tau$ and $\tau$, respectively, where $\tau\in\bZ$. In other words, if we write the obstruction bundle in the form of equivariant divisors, we have
\begin{eqnarray*}
V_D&=&R^1\pi_*(\tilde{F}^*\cO_{\bP^1_a}(\tau(ap_0-p_1))(-\mathcal{D}) )\\
V_{D_d}&=&R^1\pi_* \tilde{F}^*\cO_{\bP^1_a}(-p_{0}+\tau(p_1-ap_0)).
\end{eqnarray*}

Let
$$\xn{K}=\frac{1}{\am|\Aut(\gamma)|}\int_{[\Mb]^{\vir}}e(V).$$
Then $\xn{K}$ is a topological invariant. We will calculate $\xn{K}$ in the next section by virtual localization. Define the generating function $\xo{K}(\lambda;x)$ to be
$$\xo{K}(\lambda;x)=\sqrt{-1}^{l(\mu)-d}\sum_{\chi}\lambda^{-\chi+l(\mu)}x_\gamma\xn{K}.$$

\section{Virtual localization}
In this section, we calculate $\xn{K}$ by virtual localization. We will express $\xn{K}$ in terms of one-partition Hodge integrals and double Hurwitz numbers. Then we can obtain Theorem 2 by Burnside formula.

The computation in this section is similar to that in Appendix A of \cite{Liu-Liu-Zhou1} and can be viewed as the orbifold generalization of it. The main difference in our paper is that one needs to compute the weights of sections of orbifold line bundles over orbi-curves. These orbi-curves are constructed via root constructions which produce orbifold line bundles over them by definition. All of these involve new techniques in studying the Picard group of an algebraic stack which comes from the root construction. The readers are referred to \cite{Cad} for the general settings of the root construction and to \cite{Abr-Gra-Vis} for the application of the root construction to Gromov-Witten theory.

\subsection{Fixed points}
The connected components of the $\bC^*$ fixed points set of $\Mb$ are parameterized by labeled graphs. We first introduce some graph notations which are similar to those in \cite{Liu-Liu-Zhou1}.

Let
$$\left[ f:(C,x_1,\ldots,x_n,y_1,\cdots,y_{l(\mu)})\to \bP^1_a[m]\right]\in\Mb $$
be a fixed point of the $\bC^*$-action. The restriction of the map
$$ \tilde{f}=\pi[m]\circ f: C\to \bP^1_a$$
to an irreducible component of $C$ is either a constant map to one of the $\bC^*$
fixed points $p_0, p_1$ or a cover of $\bP^1_a$ which is fully
ramified over $p_0$ and $p_1$. We associate a labeled graph $\Gamma$ to
the $\bC^*$ fixed point
$$\left[ f:(C,x_1,\ldots,x_n,y_1,\cdots,y_{l(\mu)})\to \bP^1_a[m] \right]$$ as follows:
\begin{enumerate}
\item We assign a vertex $v$ to each connected
component $C_v$ of $\tilde{f}^{-1}(\{p_0,p_1\})$, a label
$i(v)=i$ if $\tilde{f}(C_v)=p_i$, where $i=0,1$, and a label $g(v)$
which is the arithmetic genus of $C_v$ (we define $g(v)=0$ if $C_v$ is a
point). For $i(v)=0$, we define $n(v)$ to be the number of marked points on $C_v$. Denote by $V(\Gamma)^{(i)}$ the set of vertices with $i(v)=i$,
where $i=0,1$. Then the set $V(\Gamma)$ of vertices of the graph $\Gamma$
is a disjoint union of $V(\Gamma)^{(0)}$ and $V(\Gamma)^{(1)}$. For $v\in V(\Gamma)^{(0)}$ define
\begin{eqnarray*}
r_0(v)=&2g(v)-2 + \val(v)+n(v), & v\in V(\Gamma)^{(0)},\\
\end{eqnarray*}
\begin{eqnarray*}
V^I(\Gamma)^{(0)}&=&\{v\in V(\Gamma)^{(0)}:g(v)=0,\val(v)=1,n(v)=0\},\\
V^{I,I}(\Gamma)^{(0)}&=&\{v\in V(\Gamma)^{(0)}:g(v)=0,\val(v)=1,n(v)=1\},\\
V^{II}(\Gamma)^{(0)}&=&\{v\in V(\Gamma)^{(0)}:g(v)=0,\val(v)=2,n(v)=0\},\\
V^S(\Gamma)^{(0)}&=&\{v\in V(\Gamma)^{(0)}:r_0(v)> 0\}.\\
\end{eqnarray*}

\item We assign an edge $e$ to each rational irreducible component $C_e$ of $C$
such that $\tilde{f}|_{C_e}$ is not a constant map.
Let $d(e)$ be the degree of $\tilde{f}|_{C_e}$.
Then $\tilde{f}|_{C_e}$ is fully ramified over $p_0$ and $p_1$.
Let $E(\Gamma)$ denote the set of edges of $\Gamma$.

\item The set of flags of $\Gamma$ is given by
$$F(\Gamma)=\{(v,e):v\in V(\Gamma), e\in E(\Gamma),
C_v\cap C_e\neq \emptyset \}.$$

\item For each $v\in V(\Gamma)$, define
$$d(v)=\sum_{(v,e)\in F(\Gamma)}d(e),$$
and let $\nu(v)$ be the partition of $d(v)$
determined by $\{d(e): (v,e)\in F(\Gamma)\}$ and let $\nu$ be the partition of $d$ determined by $\{d(e): e\in E(\Gamma)\}$.
When the target is $\bP^1_a[m]$, where $m>0$,
we assign an additional label for each
$v\in V(\Gamma)^{(1)}$:
let $\mu(v)$ be the partition of
$d(v)$ determined by the ramification of
$f|_{C_v}:C_v\to\bP^1_a(m)$ over $p_1^{(m)}$.
\end{enumerate}
Note that for $v\in V(\Gamma)^{(1)}$,
$\nu(v)$ coincides with the partition of $d(v)$
determined by the ramification of $f|_{C_v}:C_v\to \bP^1_a(m)$
over $p_1^{(0)}$.

Let $\cM_{\nu_i}$ be the moduli space of $\bC^*$-fixed degree $\nu_i$ covers of $\bP^1_a$ with stack structure given by $\nu_i\textrm{ (mod $a$)}$. Then the $\bC^*$-fixed locus can be identified with
$$\bigsqcup_{\chi^0+\chi^1-2l(v)=\chi}(\Md\times_{\bar{I}\cB\bZ_{a}^{l(\nu)}}\cM_{\nu_1}\times\cdots\times\cM_{\nu_{l(\nu
)}}\times\MQ//\bC^*)/\Aut(\nu),$$
where $\bar{I}\cB\bZ_{a}$ is the rigidified inertia stack of $\cB\bZ_{a}$. Therefore, we can calculate our integral over
$$\bigsqcup_{\chi^0+\chi^1-2l(v)=\chi}\Md\times\MQ//\bC^*$$
provided we include the following factor:

$$\frac{1}{|\Aut(\nu)|}\prod_{i=1}^{l(\nu)}\frac{1}{\nu_i}\frac{a}{b_i},$$
where $b_i=\frac{a}{\textrm{gcd}(a,\nu_i)}$ is the order of $\nu_i\in \bZ_a$.

\subsection{Virtual normal bundle}
Let
$$[ f:(C,x_1,\ldots,x_n,y_1,\cdots,y_{l(\mu)})\to \bP^1_a[m]]\in\Mb $$
be a fixed point of the $\bC^*$-action associated to $\Gamma$ and let $ \tilde{f}=\pi[m]\circ f: C\to \bP^1_a $. Let
$$
B_1 = \Ext^0(\Omega_C(D+D'), \cO_C),\ \ \
B_2 = H^0(C,\tilde{f}^*\cO_{\bP^1_a}(p_0)),\ \ \
B_3 = \oplus_{l=0}^{m-1} H_{\mathrm{et}}^0(\mathbf{R}_l^\bullet),
$$
$$
B_4 = \Ext^1(\Omega_C(D+D'), \cO_C),\ \ \
B_5 = H^1(C,\tilde{f}^*\cO_{\bP^1_a}(p_0)),\ \ \
B_6 = \oplus_{l=0}^{m-1} H_{\mathrm{et}}^1(\mathbf{R}_l^\bullet),
$$
where $D=y_1+\cdots+y_{l(\mu)},D'=x_1+\cdots+x_n$, and
\begin{eqnarray*}
H^0_{\mathrm{et}}(\mathbf{R}_{l}^\bullet)&\cong&  \bigoplus_{q\in
f^{-1}(p_1^{(l)})}T_q \left(f^{-1}(\bP^1_{(l)})\right)\otimes
T^*_q \left(f^{-1}(\bP^1_{(l)})\right)\cong \bC^{\oplus n_l},\\
H^1_{\mathrm{et}}(\mathbf{R}_{l}^\bullet)&\cong&
(T_{p_1^{(l)}}\bP^1_{(l)}\otimes
T_{p_1^{(l)}}\bP^1_{(l+1)})^{\oplus(n_l-1)},
\end{eqnarray*}
where $n_l$ is the number of nodes over $p_1^{(l)}$. Let $\tilde{B}_i$ denote the moving part of $B_i$ under the $\bC^*$-action. Then we have (see Appendix A in \cite{Liu-Liu-Zhou1})
$$
\frac{1}{e_T(N_\Gamma^{\mathrm{vir}} ) }
=\frac{e_T(\tilde{B}_1)e_T(\tilde{B}_5)e_T(\tilde{B}_6) }{e_T(\tilde{B}_2)e_T(\tilde{B}_4)},
$$
where $e_T$ denotes the $\bC^*$-equivariant Euler class.

Note that
$\tilde{B}_3=0$, and
$$
\tilde{B}_6=\left\{\begin{array}{ll}0,& m=0,\\
H^1_{\mathrm{et}}(\mathrm{R}_0^\bullet)=
(T_{p_1^{(0)}}\bP^1_{(0)}\otimes T_{p_1^{(0)}}\bP^1_{(1)})^{\oplus(n_0-1)}, &
m>0.
\end{array}\right.
$$

\subsubsection{The target is $\bP^1_a$}
In this case,
\begin{eqnarray*}
\tilde{B}_1&=&\bigoplus_{v\in V^I(\Gamma)^{(0)}}\left(\frac{1}{d}\right)\\
 \tB_4&=&\bigoplus_{v\in V^{II}(\Gamma)^{(0)}}\left(\sum_{(v,e)\in F(\Gamma)}\frac{\Gcd(a,d(e))}{ad(e)}\right)\oplus \bigoplus_{v\in V^S(\Gamma)^{(0)} }\left(\bigoplus_{(v,e)\in F(\Gamma)}T_{q_{(v,e)}}C_v\otimes T_{q_{(v,e)}}C_e\right).
\end{eqnarray*}
So
\begin{eqnarray*}
\frac{e_T(\tilde{B}_1)}{e_T(\tilde{B}_4)}
&=&\prod_{v\in V^I(\Gamma)^{(0)}}\frac{u}{d(v)}
\prod_{v\in V^{II}(\Gamma)^{(0)} }
   \left( \sum_{(v,e)\in F(\Gamma)}\frac{\Gcd(a,d(e))u}{ad(e)}\right)^{-1}\\
&&\cdot\prod_{v\in V^S(\Gamma)^{(0)} }
           \left(\prod_{(v,e)\in F(\Gamma)}
          \frac{1}{(\frac{u}{d(e)}-\bar{\psi}_{(v,e)})\frac{\Gcd(a,d(e))}{a}}\right).
\end{eqnarray*}
Consider the normalization sequence
\begin{eqnarray*}
0 &\to&
\tilde{f}^*\cO_{\bP^1_a}(p_0)\to
\bigoplus_{v\in V^{S}(\Gamma)^{(0)}}
(\tilde{f}|_{C_v})^*\cO_{\bP^1_a}(p_0)
\oplus  \bigoplus_{e\in E(\Gamma)}(\tilde{f}|_{C_e})^*\cO_{\bP^1_a}(p_0)\\
  &\to& \bigoplus_{v\in V^{II}(\Gamma)^{(0)}}(\tilde{f}|_{q_v})^* \cO_{\bP^1_a}(p_0)
\oplus\bigoplus_{v\in V^S(\Gamma)^{(0)}}\left(\bigoplus_{(v,e)\in F}(\tilde{f}|_{q_{(v,e)}})^*
     \cO_{\bP^1_a}(p_0)\right)\to 0,
\end{eqnarray*}
where $q_v$ and $ q_{(v,e)}$ are the corresponding nodes.
The corresponding long exact sequence reads
\begin{eqnarray*}
0& \to & H^0(C,\tilde{f}^*\cO_{\bP^1_a}(p_0))
  \to  \bigoplus_{v\in V^{S}(\Gamma)^{(0)} }
         H^0(C_v,(\tilde{f}|_{C_v})^*\cO_{\bP^1_a}(p_0))
  \oplus \bigoplus_{e\in E(\Gamma)}
\left(\bigoplus_{l=0}^{[\frac{d(e)}{a}]}\left(\frac{l}{d(e)}\right)\right) \\
&\to&\bigoplus_{v\in V^{II}(\Gamma)^{(0)} } H^0(q_v,(\tilde{f}|_{q_v})^* \cO_{\bP^1_a}(p_0))
\oplus\bigoplus_{v\in V^S(\Gamma)^{(0)} }\left(\bigoplus_{(v,e)\in F}
     H^0(q_{(v,e)},(\tilde{f}|_{q_{(v,e)}})^*\cO_{\bP^1_a}(p_0))\right) \\
 &\to & H^1(C,\tilde{f}^*\cO_{\bP^1_a}(p_0))
 \to  \bigoplus_{v\in V^{S}(\Gamma)^{(0)} }
         H^1(C_v,(\tilde{f}|_{C_v})^*\cO_{\bP^1_a}(p_0)) \to 0.
\end{eqnarray*}
The term $H^0(C_v,(\tilde{f}|_{C_v})^*\cO_{\bP^1_a}(p_0))$ contributes 0 unless all monodromies around loops on $C_v$ are trivial. The term $H^0(q_v,(\tilde{f}|_{q_v})^* \cO_{\bP^1_a}(p_0))$ (resp. $H^0(q_{(v,e)},(\tilde{f}|_{q_{(v,e)}})^*\cO_{\bP^1_a}(p_0))$) contributes zero unless $q_v$ (resp. $q_{(v,e)}$) is not a stack point. So
$$
\frac{e_T(\tilde{B}_5)}{e_T(\tilde{B}_2)}=
\prod_{v\in V(\Gamma)^{(0)}}
\left(\Lambda_{g(v)}^{\vee,U}(\frac{u}{a})(\frac{u}{a})^{\sum_{(v,e)\in F(\Gamma)}\delta_{0,\langle \frac{d(e)}{a}\rangle}-\delta_v}  \right)
\prod_{e\in E(\Gamma)}\left(\frac{d(e)^{[\frac{d(e)}{a}]}}{[\frac{d(e)}{a}]!} u^{-[\frac{d(e)}{a}]}\right),
$$
where $\langle x\rangle=x-[x]$ and
$$\delta_{v}=\left\{\begin{array}{ll}1, &\textrm{if all monodromies around loops on $C_v$ are trivial},\\
0, &\textrm{otherwise}.\end{array} \right.$$
So we have the following Feynman rules:
$$\frac{1}{e_T(N_{\Gamma}^{\mathrm{vir} })}
= \prod_{v \in V(\Gamma)^0} A_v
\prod_{e \in E(\Gamma)} A_e,
$$
where
\begin{eqnarray*}
A_v &= &\left\{
\begin{array}{ll}
\Lambda_{g(v)}^{\vee,U}(\frac{u}{a})(\frac{u}{a})^{-\delta_v}\prod_{(v,e)\in F(\Gamma)}
\frac{(\frac{u}{a})^{\delta_{0,\langle \frac{d(e)}{a}\rangle}}}{(\frac{u}{d(e)} - \bar{\psi}_{(v, e)})\frac{\Gcd(a,d(e))}{a}},
   & v \in V^S(\Gamma)^{(0)}, \\
\frac{u}{d(v)}, & v \in V^I(\Gamma)^{(0)}, \\
1, & v \in V^{I,I}(\Gamma)^{(0)},\\
\frac{1}{\frac{\Gcd(a,d(e_1))}{d(e_1)} + \frac{\Gcd(a,d(e_2))}{d(e_2)}}(\frac{u}{a})^{-1+\delta_{0,\langle \frac{d(e_1)}{a}\rangle}},
   & v \in V^{II}(\Gamma)^{(0)},\\
&(v, e_1), (v, e_2) \in F(\Gamma),
\end{array} \right.\\
 A_e&=& \frac{d(e)^{[\frac{d(e)}{a}]}}{[\frac{d(e)}{a}]!} u^{-[\frac{d(e)}{a}]}.
\end{eqnarray*}

\subsubsection{The target is $\bP^1_a[m]$, $m>0$}
A similar computation shows that
\begin{eqnarray*}
\frac{e_T(\tilde{B}_1)}{e_T(\tilde{B}_4)}
&=&\prod_{v\in V^I(\Gamma)^{(0)}}\frac{u}{d(v)}
\prod_{v\in V^{II}(\Gamma)^{(0)} }
   \left( \sum_{(v,e)\in F(\Gamma)}\frac{\Gcd(a,d(e))u}{ad(e)}\right)^{-1}\\
&&\cdot\prod_{v\in V^S(\Gamma)^{(0)} }
           \left(\prod_{(v,e)\in F(\Gamma)}
          \frac{1}{(\frac{u}{d(e)}-\bar{\psi}_{(v,e)})\frac{\Gcd(a,d(e))}{a}}\right)
          \prod_{v\in V(\Gamma)^{(1)} }
        \left(\prod_{(v,e)\in F(\Gamma)}
          \frac{1}{\frac{-u}{d(e)}-\psi_{(v,e)}}\right).
\end{eqnarray*}
For $\frac{e_T(\tilde{B}_5)}{e_T(\tilde{B}_2)}$, we also have a similar long exact sequence
\begin{eqnarray*}
0& \to & H^0(C,\tilde{f}^*\cO_{\bP^1_a}(p_0))\\
  &\to &  \bigoplus_{v\in V^{S}(\Gamma)^{(0)} }
         H^0(C_v,(\tilde{f}|_{C_v})^*\cO_{\bP^1_a}(p_0))
  \oplus \bigoplus_{v\in V(\Gamma)^{(1)} }H^0(C_v,\cO_{C_v})\otimes (0)
  \oplus \bigoplus_{e\in E(\Gamma)}
\left(\bigoplus_{l=0}^{[\frac{d(e)}{a}]}\left(\frac{l}{d(e)}\right)\right) \\
&\to&\bigoplus_{v\in V^{II}(\Gamma)^{(0)} } H^0(q_v,(\tilde{f}|_{q_v})^* \cO_{\bP^1_a}(p_0))
\oplus\bigoplus_{v\in V^S(\Gamma)^{(0)} }\left(\bigoplus_{(v,e)\in F}
     H^0(q_{(v,e)},(\tilde{f}|_{q_{(v,e)}})^*\cO_{\bP^1_a}(p_0))\right)
     \oplus\\
&&\bigoplus_{v\in V(\Gamma)^{(1)} }\left(\bigoplus_{(v,e)\in F}(0)\right) \\
 &\to & H^1(C,\tilde{f}^*\cO_{\bP^1_a}(p_0))\\
 &\to & \bigoplus_{v\in V^{S}(\Gamma)^{(0)} }
         H^1(C_v,(\tilde{f}|_{C_v})^*\cO_{\bP^1_a}(p_0))
         \oplus  \bigoplus_{v\in  V(\Gamma)^{(1)} }
H^1(C_v,\cO_{C_v})\otimes (0)\to 0.
\end{eqnarray*}
So $\frac{e_T(\tilde{B}_5)}{e_T(\tilde{B}_2)}$ has the same expression as in the previous case.
Finally,
$$
\tilde{B}_6=\left(T_{p_1^{(0)}}\bP^1_{(0)}
             \otimes T_{p_1^{(0)}}\bP^1_{(1)}\right)^{|E(\Gamma)|-1},
$$
so
$$
e_T(\tilde{B}_6)=(-u-\psi^0)^{|E(\Gamma)|-1},
$$
where $\psi^0$ is the target $\psi$ class as in section 3.3. Note that $\psi^0=d(e)\psi_{(v,e)}$ for $v\in V(\Gamma)^{(1)}$, $(v,e)\in F$.

Therefore, we have the following Feynman rules:
$$\frac{1}{e_T(N_{\Gamma}^{\mathrm{vir} })}
= \frac{1}{-u-\psi^0}\prod_{v \in V(\Gamma)} A_v
\prod_{e \in E(\Gamma)} A_e,
$$
where
\begin{eqnarray*}
A_v &= &\left\{
\begin{array}{ll}
\prod_{(v,e)\in F(\Gamma)}\frac{-u-\psi^0}{\frac{-u}{d(e)}-\psi_{(v,e)} }
 =\prod_{(v,e)\in F(\Gamma)} d(e), &
 v \in V(\Gamma)^{(1)}\\
\Lambda_{g(v)}^{\vee,U}(\frac{u}{a})(\frac{u}{a})^{-\delta_v}\prod_{(v,e)\in F(\Gamma)}
\frac{(\frac{u}{a})^{\delta_{0,\langle \frac{d(e)}{a}\rangle}}}{(\frac{u}{d(e)} - \bar{\psi}_{(v, e)})\frac{\Gcd(a,d(e))}{a}},
   & v \in V^S(\Gamma)^{(0)}, \\
\frac{u}{d(v)}, & v \in V^I(\Gamma)^{(0)}, \\
1, & v \in V^{I,I}(\Gamma)^{(0)},\\
\frac{1}{\frac{\Gcd(a,d(e_1))}{d(e_1)} + \frac{\Gcd(a,d(e_2))}{d(e_2)}}(\frac{u}{a})^{-1+\delta_{0,\langle \frac{d(e_1)}{a}\rangle}},
   & v \in V^{II}(\Gamma)^{(0)},\\
&(v, e_1), (v, e_2) \in F(\Gamma),
\end{array} \right.\\
 A_e&=& \frac{d(e)^{[\frac{d(e)}{a}]}}{[\frac{d(e)}{a}]!} u^{-[\frac{d(e)}{a}]}.
\end{eqnarray*}

\subsection{The bundle $V_{D}$}\label{D}
The short exact sequence
$$
0\to\cO_C(-D)\to \cO_C \to \cO_D\to 0
$$
gives rise to a long exact sequence
\begin{eqnarray*}
0&\to& H^0(C,\cO_C(-D) )\to H^0(C,\cO_C)\to
\bigoplus_{i=1}^{l(\mu)}\cO_{y_i}\\
&\to& H^1(C,\cO_C(-D) )\to H^1(C,\cO_C)\to 0.
\end{eqnarray*}
The representations of $\bC^*$ are
$$
0\to\bigoplus_{i=1}^{k}(\tau)\to \bigoplus_{i=1}^{l(\mu)}(\tau) \to
H^1(C,\cO_C(-D))\otimes (\tau)\to H^1(C,\cO_C)\otimes(\tau)\to 0,
$$
where $k$ is the number of connected components of $C$.
So
$$
e_T(V_D)=e_T(V_0)(\tau u)^{l(\mu)-k},
$$
where
$$V_0=R^1\pi_*\cO_{\mathcal{U}}.$$

\subsubsection{The target is $\bP^1_a$}
In this case, the number of connected components of $C$ is $|V(\Gamma)^0|$. Consider the normalization sequence
$$
0\to \cO_C\to\bigoplus_{v\in V(\Gamma)^0}\cO_{C_{v}}\oplus \bigoplus_{i=1}^{l(\mu)} \cO_{C_{e_i}}
 \to \bigoplus_{i=1}^{l(\mu)}\cO_{q_i}\to 0.
$$
The corresponding $\bC^*$-representation of the long exact sequence reads
\begin{eqnarray*}
0& \to& \bigoplus_{v\in V(\Gamma)^0}  (\tau)
   \to \bigoplus_{v\in V(\Gamma)^0}(\tau)\oplus
      \bigoplus_{i=1}^{l(\mu)} (\tau)
   \to\bigoplus_{i=1}^{l(\mu)} (\tau)\\
& \to & H^1(C,\cO_C)\otimes(\tau)
  \to \bigoplus_{v\in V(\Gamma)^0}H^1(C_{v},\cO_{C_{v}})\otimes(\tau)\to 0.
\end{eqnarray*}
So
$$
i_\Gamma^* e_T(V_0)=\prod_{v\in V(\Gamma)^0}\Lambda_{g(v)}^{\vee,1}(\tau u),
$$
where 1 denotes the trivial representation of $\bZ_a$.

Therefore, we have
$$
i_{\Gamma}^* e_T(V_D)=\prod_{v\in V(\Gamma)^0}A^D_{v},
$$
where
$$
A_{v}^D = \Lambda_{g(v)}^{\vee,1}(\tau u) \cdot
(\tau u)^{\val(v)-1}.
$$

\subsubsection{The target is $\bP^1_a[m]$, $m>0$}
Consider the normalization sequence
\begin{eqnarray*}
0 &\to&\cO_C\to
\bigoplus_{v\in V^{S}(\Gamma)^{(0)}\cup V(\Gamma)^{(1)} }\cO_{C_v}
 \oplus \bigoplus_{e\in E(\Gamma)}\cO_{C_e}\\
&\to&\bigoplus_{v\in V^{II}(\Gamma)^{(0)}} \cO_{q_v}
\oplus\bigoplus_{v\in V^S(\Gamma)^{(0)}\cup V(\Gamma)^{(1)}}\left(\bigoplus_{(v,e)\in F}
     \cO_{q_{(v,e)}}\right)\to 0.
\end{eqnarray*}
The corresponding $\bC^*$-representation of the long exact sequence reads
\begin{eqnarray*}
0& \to &\bigoplus_{i=1}^{k}(\tau)
\to \bigoplus_{v\in V^{S}(\Gamma)^{(0)}\cup V(\Gamma)^{(1)} }(\tau)
     \oplus \bigoplus_{e\in E(\Gamma)}(\tau)\\
&\to&\bigoplus_{v\in V^{II}(\Gamma)^{(0)}} (\tau)
\oplus\bigoplus_{v\in V^S(\Gamma)^{(0)}\cup V(\Gamma)^{(1)}}\left(\bigoplus_{(v,e)\in F}
     (\tau)\right)\\
 &\to & H^1(C,\cO_C)\otimes(\tau)
\to \bigoplus_{v\in V^{S}(\Gamma)^{(0)}\cup V(\Gamma)^{(1)} }
           H^1(C_v,\cO_{C_v})\otimes(\tau)\to 0,
\end{eqnarray*}
where $k$ is the number of connected components of $C$.
So
\begin{eqnarray*}
i_\Gamma^* e_T(V_0)&=&(\tau u)^{|E(\Gamma)|-|V(\Gamma)|+k}
\prod_{v\in V(\Gamma)}\Lambda_{g(v)}^{\vee,1}(\tau u)\\
i_\Gamma^* e_T(V_D)&=&(\tau u)^{|E(\Gamma)|-|V(\Gamma)|+l(\mu)}
\prod_{v\in V(\Gamma)}\Lambda_{g(v)}^{\vee,1}(\tau u).\\
\end{eqnarray*}
We have the following Feynman rules:
$$i_{\Gamma}^*e_T(V_D) = \prod_{v\in V(\Gamma)} A_v^D,$$
where
$$
 A_v^{D} = \begin{cases}
\Lambda_{g(v)}^{\vee,1}(\tau u) \cdot (\tau u)^{\val(v)-1},
& v \in V(\Gamma)^{(0)}, \\
\Lambda_{g(v)}^{\vee,1}(\tau u) \cdot
(\tau u)^{l(\mu(v))-1}
& v \in V(\Gamma)^{(1)}.
\end{cases} \\
$$

\subsection{The bundle $V_{D_d}$} \label{Dd}
\subsubsection{The target is $\bP^1_a$}
In this case, $k=|V(\Gamma)^0|$. Consider the normalization sequence
\begin{eqnarray*}
0 &\to&
\tilde{f}^*\cO_{\bP^1_a}(-p_0)\to
\bigoplus_{v\in V^{S}(\Gamma)^{(0)}}
(\tilde{f}|_{C_v})^*\cO_{\bP^1_a}(-p_0)
\oplus  \bigoplus_{e\in E(\Gamma)}(\tilde{f}|_{C_e})^*\cO_{\bP^1_a}(-p_0)\\
  &\to& \bigoplus_{v\in V^{II}(\Gamma)^{(0)}}(\tilde{f}|_{q_v})^* \cO_{\bP^1_a}(-p_0)
\oplus\bigoplus_{v\in V^S(\Gamma)^{(0)}}\left(\bigoplus_{(v,e)\in F}(\tilde{f}|_{q_{(v,e)}})^*
     \cO_{\bP^1_a}(-p_0)\right)\to 0.
\end{eqnarray*}
The corresponding long exact sequence reads
\begin{eqnarray*}
0& \to & H^0(C,\tilde{f}^*\cO_{\bP^1_a}(-p_0))
  \to  \bigoplus_{v\in V^{S}(\Gamma)^{(0)} }
         H^0(C_v,(\tilde{f}|_{C_v})^*\cO_{\bP^1_a}(-p_0)) \\
&\to&\bigoplus_{v\in V^{II}(\Gamma)^{(0)} } H^0(q_v,(\tilde{f}|_{q_v})^* \cO_{\bP^1_a}(-p_0))
\oplus\bigoplus_{v\in V^S(\Gamma)^{(0)} }\left(\bigoplus_{(v,e)\in F}
     H^0(q_{(v,e)},(\tilde{f}|_{q_{(v,e)}})^*\cO_{\bP^1_a}(-p_0))\right) \\
 &\to & H^1(C,\tilde{f}^*\cO_{\bP^1_a}(-p_0))\\
 &\to & \bigoplus_{v\in V^{S}(\Gamma)^{(0)} }
         H^1(C_v,(\tilde{f}|_{C_v})^*\cO_{\bP^1_a}(-p_0))\oplus \bigoplus_{e\in E(\Gamma)}
\left(\bigoplus_{l=1}^{[\frac{d(e)}{a}]-\delta_{0,\langle \frac{d(e)}{a}\rangle}}\left(-\tau-\frac{l}{d(e)}\right)\right) \to 0.
\end{eqnarray*}
We have the following Feynman rules:
$$i_{\Gamma}^*e_T(V_{D_d}) = \prod_v A_v^{D_d} \cdot
\prod_{e\in E(\Gamma)} A_e^{D_d},$$
where
\begin{eqnarray*}
A_v &= &\left\{
\begin{array}{ll}
\Lambda_{g(v)}^{\vee,U^\vee}((-\tau-\frac{1}{a})u)((-\tau-\frac{1}{a})u)^{-\delta_v}\\ \prod_{(v,e)\in F(\Gamma)}
((-\tau-\frac{1}{a})u)^{\delta_{0,\langle \frac{d(e)}{a}\rangle}},
    &v \in V^S(\Gamma)^{(0)}, \\
1, & v \in V^I(\Gamma)^{(0)}, \\
1, & v \in V^{I,I}(\Gamma)^{(0)},\\
\left((-\tau-\frac{1}{a})u\right)^{\delta_{0,\langle \frac{d(e_1)}{a}\rangle}},
    &v \in V^{II}(\Gamma)^{(0)},(v, e_1),(v,e_2) \in F(\Gamma),
\end{array} \right.\\
 A_e&=& \frac{\prod_{l=1}^{[\frac{d(e)}{a}]-\delta_{0,\langle \frac{d(e)}{a}\rangle}}(d(e)\tau+l)}{d(e)^{[\frac{d(e)}{a}]-\delta_{0,\langle \frac{d(e)}{a}\rangle}}} (-u)^{[\frac{d(e)}{a}]-\delta_{0,\langle \frac{d(e)}{a}\rangle}}.
\end{eqnarray*}

\subsubsection{The target is $\bP^1_a[m]$, $m>0$}
We have a similar long exact sequence
\begin{eqnarray*}
0& \to & H^0(C,\tilde{f}^*\cO_{\bP^1_a}(-p_0))\\
  &\to & \bigoplus_{v\in V^{S}(\Gamma)^{(0)} }
         H^0(C_v,(\tilde{f}|_{C_v})^*\cO_{\bP^1_a}(-p_0))\oplus \bigoplus_{v\in V(\Gamma)^{(1)}}H^0(C_v,\cO_{C_v})\otimes (-\tau) \\
&\to&\bigoplus_{v\in V^{II}(\Gamma)^{(0)} } H^0(q_v,(\tilde{f}|_{q_v})^* \cO_{\bP^1_a}(-p_0))
\oplus\bigoplus_{v\in V^S(\Gamma)^{(0)} }\left(\bigoplus_{(v,e)\in F}
     H^0(q_{(v,e)},(\tilde{f}|_{q_{(v,e)}})^*\cO_{\bP^1_a}(-p_0))\right)\\
  &&   \oplus \bigoplus_{v\in V(\Gamma)^{(1)}}\left(\bigoplus_{(v,e)\in F}
     (-\tau)\right) \\
 &\to & H^1(C,\tilde{f}^*\cO_{\bP^1_a}(-p_0))\\
 &\to & \bigoplus_{v\in V^{S}(\Gamma)^{(0)} }
         H^1(C_v,(\tilde{f}|_{C_v})^*\cO_{\bP^1_a}(-p_0))\oplus \bigoplus_{v\in V(\Gamma)^{(1)}}H^1(C_v,\cO_{C_v})\otimes (-\tau)\\
 &&        \oplus \bigoplus_{e\in E(\Gamma)}
\left(\bigoplus_{l=1}^{[\frac{d(e)}{a}]-\delta_{0,\langle \frac{d(e)}{a}\rangle}}\left(-\tau-\frac{l}{d(e)}\right)\right) \to 0.
\end{eqnarray*}
So we have the following Feynman rules:
$$i_{\Gamma}^*e_T(V_{D_d}) = \prod_v A_v^{D_d} \cdot
\prod_{e\in E(\Gamma)} A_e^{D_d},$$
where
\begin{eqnarray*}
A_v &= &\left\{
\begin{array}{ll}
\Lambda_{g(v)}^{\vee,U^\vee}((-\tau-\frac{1}{a})u)((-\tau-\frac{1}{a})u)^{-\delta_v}\\ \prod_{(v,e)\in F(\Gamma)}
((-\tau-\frac{1}{a})u)^{\delta_{0,\langle \frac{d(e)}{a}\rangle}},
    &v \in V^S(\Gamma)^{(0)}, \\
1, & v \in V^I(\Gamma)^{(0)}, \\
1, & v \in V^{I,I}(\Gamma)^{(0)},\\
\left((-\tau-\frac{1}{a})u\right)^{\delta_{0,\langle \frac{d(e_1)}{a}\rangle}},
    &v \in V^{II}(\Gamma)^{(0)},(v, e_1),(v,e_2) \in F(\Gamma),\\
    \Lambda_{g(v)}^{\vee,1}(-\tau u) \cdot
(-\tau u)^{\val(v)-1}, & v \in V(\Gamma)^{(1)},
\end{array} \right.\\
 A_e&=& \frac{\prod_{l=1}^{[\frac{d(e)}{a}]-\delta_{0,\langle \frac{d(e)}{a}\rangle}}(d(e)\tau+l)}{d(e)^{[\frac{d(e)}{a}]-\delta_{0,\langle \frac{d(e)}{a}\rangle}}} (-u)^{[\frac{d(e)}{a}]-\delta_{0,\langle \frac{d(e)}{a}\rangle}}.
\end{eqnarray*}
\\
\subsection{Contribution from each graph}
By Mumford's relation \cite{Bry-Gra-Pan}, we have
$$\Lambda_{g(v)}^{\vee,1}(-\tau u)\Lambda_{g(v)}^{\vee,1}(\tau u)=(-1)^{g(v)}(\tau u)^{2g(v)}, v \in V(\Gamma)^{(1)}.$$
So we have
$$\frac{i_{\Gamma}^*e_T(V)}{e_T(N_{\Gamma}^{\mathrm{vir} })}=A^{0}A^{1},$$
where
\begin{eqnarray*}
A^0&=&a_{\nu}\sqrt{-1}^{l(\nu)-|\nu|}\prod_{i=1}^{l(\nu)}\frac{\prod_{l=1}^{[\frac{\nu_i}{a}]}(\nu_{i}\tau+l)}{[\frac{\nu_i}
{a}]!(u-\nu_{i}\bar{\psi_i})\frac{\Gcd(\nu_{i},a)}{a}}\\
&&\cdot\prod_{v \in V(\Gamma)^{(0)}}(-1)^{\sum_{(v,e)\in F(\Gamma)}[\frac{d(e)}{a}]-\delta_v}\Lambda_{g(v)}^{\vee,U}(\frac{u}{a})\Lambda_{g(v)}^{\vee,U^\vee}((-\tau-\frac{1}{a})u)
\Lambda_{g(v)}^{\vee,1}(\tau u)\\
&&\cdot\left(\frac{u}{a}\right)^{\sum_{(v,e)\in F(\Gamma)}\delta_{0,\langle \frac{d(e)}{a}\rangle}-\delta_v}\left((\tau+\frac{1}{a})u\right)^{-\delta_v}(\tau u)^{\val(v)-1}\\
A^1&=& \left\{\begin{array}{ll}\sqrt{-1}^{|\mu|-l(\mu)}, &\textrm{the target is }\bP^1_a,\\
\sqrt{-1}^{|\nu|-l(\mu)}a_{\nu}\frac{(\sqrt{-1}\tau u)^{-\chi^{1}+l(\mu)+l(\nu)}}{-u-\psi^0},  &\textrm{the target is }\bP^1_a[m],m>0.
\end{array}\right.
\end{eqnarray*}
\\
\subsection{Proof of Theorem 2}
\begin{eqnarray*}
&&\xn{K}\\
&=&\frac{1}{\am|\Aut(\gamma)|}\int_{[\Mb]^{\vir}}e(V)\\
&=&\frac{1}{\am|\Aut(\gamma)|}\sum_{\chi^0+\chi^1-2l(v)=\chi}\frac{1}{|\Aut(\nu)|}\prod_{i=1}^{l(\nu)}\frac{1}{\nu_i}
\frac{a}{\frac{a}{\textrm{gcd}(a,\nu_i)}}\\
&&\cdot\int_{[\Md\times\MQ//\bC^*]^{\vir}}\frac{i_{\Gamma}^*e_T(V)}{e_T(N_{\Gamma}^{\mathrm{vir} })}\\
&=&\sqrt{-1}^{|\nu|-l(\mu)}\left(\sum_{\chi^0+\chi^1-2l(v)=\chi}G^\bullet_{\chi^0,\nu,\gamma}(\tau)_{a}
\cdot z_{\nu}
\frac{(-\sqrt{-1}\tau)^{-\chi^1 +l(\nu) +l(\mu)}}
{(-\chi^1 +l(\nu) +l(\mu))!}H^\bullet_{\chi^1,\nu,\mu}\right).
\end{eqnarray*}

Recall that
$$\xo{K}(\lambda;x)=\sqrt{-1}^{l(\mu)-d}\sum_{\chi}\lambda^{-\chi+l(\mu)}x_\gamma\xn{K}.$$
We have
$$\xo{K}(\lambda;x)=\sum_{|\nu|=|\mu|}G^\bullet_{\nu}(\lambda;\tau;x)_{a}z_{\nu}
\Phi^\bullet_{\nu,\mu}(-\sqrt{-1}\tau\lambda).$$
This finishes the proof of Theorem 2.

In the proof of Theorem 2, we used the following convention for the unstable integrals:

\begin{eqnarray*}
\int_{\Mbar_{0,(0)}(\cB\bZ_a)}\frac{1}{1-d\bar{\psi}}&=&\frac{1}{ad^2}\\
\int_{\Mbar_{0,(c,-c)}(\cB\bZ_a)}\frac{1}{(1-\mu_1\bar{\psi}_1)(1
-\mu_2\bar{\psi_2})}
&=& \frac{1}{a(\mu_1+\mu_2)}\\
\int_{\Mbar_{0,(c,-c)}(\cB\bZ_a)}\frac{1}{(1-d\bar{\psi}_1)}=\frac{1}{ad}.
\end{eqnarray*}

\subsection{Abelian groups}
We generalize the Mari\~{n}o-Vafa formula to arbitrary finite abelian groups. Let $G$ be an finite abelian group and let $R$ be an irreducible representation
$$\phi^{R}:G\to \bC^*$$
with associated short exact sequence
$$0\to K\to G\stackrel{\phi^{R}}{\to}\textrm{Im}(\phi^{R})\cong \bZ_a\to 0.$$
Then $\phi^{R}$ induces a morphism
$$\rho:\Mbar_{g, \gamma}(\cB G)\to \Mbar_{g, \phi^{R}(\gamma)}(\cB\bZ_{a}).$$
The following results are shown in \cite{JPT}.
\begin{lemma}
\begin{eqnarray*}
\bE^R&\cong &\rho^*(\bE^U)\\
 \textrm{deg}(\rho)&=&\left\{\begin{array}{ll}0,&\sum_{i=1}^{n}\gamma_{i}\neq 0,\\
|K|^{2g-1},&\sum_{i=1}^{n}\gamma_{i}=0,\end{array}\right.
\end{eqnarray*}
where $\bE^R$ (resp. $\bE^U$) denotes the Hodge bundle on $\Mbar_{g, \gamma}(\cB G)$ (resp. $\Mbar_{g, \gamma}(\cB \bZ_a)) $ corresponding to the representation $R$ (resp. $U$).
\end{lemma}
Let $\gamma=(\gamma_{1}, \cdots, \gamma_{n})$ be a vector of elements in $G$ such that $\phi^{R}(\gamma)$ is a vector of \emph{nontrivial} elements in $\bZ_a$. Let $\bar{\mu}=(\bar{\mu}_{1},\cdots,\bar{\mu}_{l(\bar{\mu})})$ be a vector of elements in $G$ such that
$$\phi^{R}(\bar{\mu})=\mu.$$
Here we view $\mu$ as a vector of elements in $\bZ_a$. $\bar{\mu}$ is required to satisfy conditions (1), (2) in section 1.2 which imply that $\Mbar_{g, \gamma-\bar{\mu}}(\cB G)$ and $\Mbar_{g, \phi^{R}(\gamma)-\mu}(\cB\bZ_a)$ are both empty or both nonempty. Then we have
\begin{eqnarray*}
&&\int_{\Mbar_{g, \gamma-\bar{\mu}}(\cB G)}\frac{\left(-\frac{1}{a}(\tau+\frac{1}{a})\right)^{-\delta_K}
\Lambda_{g}^{\vee,R}(\frac{1}{a})\Lambda_{g}^{\vee,R^\vee}(-\tau-\frac{1}{a})
\Lambda_{g}^{\vee,1}(\tau )}{\prod_{i=1}^{l(\mu)}(1-\mu_{i}\bar{\psi}_{i})}\\
&=&|K|^{2g-1}\int_{\Mbar_{g, \phi^R(\gamma)-\mu}(\cB\bZ_a)}\frac{\left(-\frac{1}{a}(\tau+\frac{1}{a})\right)^{-\delta}
\Lambda_{g}^{\vee,U}(\frac{1}{a})\Lambda_{g}^{\vee,U^\vee}(-\tau-\frac{1}{a})
\Lambda_{g}^{\vee,1}(\tau )}{\prod_{i=1}^{l(\mu)}(1-\mu_{i}\bar{\psi}_{i})}.
\end{eqnarray*}
Therefore
\begin{eqnarray*}
G_{g,\mu,\gamma}(\tau)_{a,\phi^R}&=&|K|^{2g-1}G_{g,\mu,\phi^{R}(\gamma)}(\tau)_{a},\\
G_{\mu,\gamma}(\lambda;\tau)_{a,\phi^R}&=&\sum_{g=0}^{\infty}
\lambda^{2g-2+l(\mu)}|K|^{2g-1}G_{g,\mu,\phi^{R}(\gamma)}(\tau)_{a}\\
&=&\frac{|K|}{|K|^{l(\mu)}}G_{\mu,\phi^R(\gamma)}(|K|\lambda;\tau)_{a},\\
G_\gamma(\lambda;\tau;p)_{a,\phi^R}&=&\sum_{\mu\neq \emptyset} \frac{|K|}{|K|^{l(\mu)}}G_{\mu,\phi^R(\gamma)}(|K|\lambda;\tau)_{a}p_{\mu}\\
&=&|K|G_{\phi^R(\gamma)}(|K|\lambda;\tau;\frac{p_1}{|K|},\frac{p_2}{|K|},\cdots)_{a}\\
&=&|K|R_{\phi^R(\gamma)}(|K|\lambda;\tau;\frac{p_1}{|K|},\frac{p_2}{|K|},\cdots)_{a}\\
&=&|K|R_{\phi^R(\gamma)}(|K|\lambda;\tau;\frac{p}{|K|})_{a}.
\end{eqnarray*}
This finishes the proof of Theorem 3.

\section{Application to the local Gromov-Witten theory of orbi-curves}
\subsection{Definitions and some known facts}
Let $(X,p_1,\cdots,p_r,q_1,\cdots,q_s)$ be a fixed non-singular genus $g$ orbi-curve with stack points $p_1,\cdots,p_r$ of orders $a_1,\cdots,a_r$ and with ordinary points $q_1,\cdots,q_s$. Let $\alpha^1,\cdots,\alpha^s$ be partitions of $d>0$. Let $\gamma^1,\cdots,\gamma^r$ be vectors of nontrivial elements in $\bZ_{a_1},\cdots,\bZ_{a_r}$ respectively. Then there is a corresponding moduli space of relative stable maps
$$\Mbar^\bullet_\gamma(X,(\alpha^1,\cdots,\alpha^s))$$
parameterizing stable maps from possibly disconnected curves to $X$ ramified over $q_i$ with ramification type $\alpha^i$ and with marked points with monodromies $\gamma=\gamma^1+\cdots+\gamma^r$.

Let
$$\pi:\cU\to \Mbar^\bullet_\gamma(X,(\alpha^1,\cdots,\alpha^s))$$
be the universal domain curve and let
$$P:\cT\to\Mbar^\bullet_\gamma(X,(\alpha^1,\cdots,\alpha^s))$$
be the universal target. Then there is an evaluation map
$$F:\cU\to \cT.$$
Let $\cQ$ be the universal prescribed branch divisor and let $\cD$ be the universal ramification divisor. Let
$$I(X,\overrightarrow{\alpha})=R^1\pi_{*}(F^*(\omega(\cQ))\oplus \cO(-\cD))-R^0\pi_{*}(F^*(\omega(\cQ))\oplus \cO(-\cD))$$
where $\omega$ is the relative dualizing sheaf of $P$, $\overrightarrow{\alpha}=(\alpha^1,\cdots,\alpha^s)$. Then the \emph{relative local invariants} are defined by
$$Z^{b,\gamma}_d(g)_{\overrightarrow{\alpha},\overrightarrow{a}}=\frac{1}{|\Aut(\gamma)||\Aut(\alpha^1)|\cdots|\Aut(\alpha^s)|}
\int_{[\Mbar^\bullet_\gamma(X,\overrightarrow{\alpha})]^{\vir}}c_b(I(X,\overrightarrow{\alpha}))$$
and their corresponding generating functions are defined by
$$Z_d(g)(\lambda;x)_{\overrightarrow{\alpha},\overrightarrow{a}}=
\sum_{b,\gamma}Z^{b,\gamma}_d(g)_{\overrightarrow{\alpha},\overrightarrow{a}}\lambda^bx_\gamma$$
where $\overrightarrow{a}=(a_1,\cdots,a_r)$. The only difference between our definition and the non-orbifold definition in \cite{Bry-Pan1} is that there are two additional indices $\overrightarrow{a}$ and $\gamma$ in our definition. Note that $I(X,\overrightarrow{\alpha})$ has rank $b$ when restricted to the components of $\Mbar^\bullet_\gamma(X,\overrightarrow{\alpha})$ with virtual dimension $b$.

We use the notation $\mu\vdash d$ to indicate  $\mu$ is a partition of $d$.
The following proposition is the orbifold version of the gluing law shown in \cite{Bry-Pan1}.
\\
\begin{proposition}\label{gluing}
Let $\overrightarrow{\alpha}=(\alpha^1,\cdots,\alpha^s)$ and $\overrightarrow{a}=(a_1,\cdots,a_r)$. For any choice $g=g_1+g_2$ and any splitting
\begin{eqnarray*}
\{\alpha^1,\cdots,\alpha^s\}&=&\{\alpha^1,\cdots,\alpha^k\}\cup \{\alpha^{k+1},\cdots,\alpha^s\}\\
\{a_1,\cdots,a_r\}&=&\{a_1,\cdots,a_l\}\cup \{a_{l+1},\cdots,a_r\}
\end{eqnarray*}
we have
$$Z_d(g)(\lambda;x)_{\overrightarrow{\alpha},\overrightarrow{a}}=\sum_{\mu\vdash d}z_\mu Z_d(g_1)(\lambda;x)_{(\alpha^1,\cdots,\alpha^k,\mu),(a_1,\cdots,a_l)}
Z_d(g_2)(\lambda;x)_{(\alpha^{k+1},\cdots,\alpha^s,\mu),(a_{l+1},\cdots,a_r)}.$$
\end{proposition}
The only difference between Proposition \ref{gluing} and the gluing law in \cite{Bry-Pan1} is that there is an additional splitting of $\overrightarrow{a} =(a_1,\cdots,a_r)$. The proof of Proposition \ref{gluing} is the same as that of the gluing law in \cite{Bry-Pan1}.

\subsection{Calculation of $Z_d(0)(\lambda;x)_{(\mu),(a)}$}

In \cite{Bry-Pan2}, the local Gromov-Witten theory of ordinary curves is solved even without imposing the Calabi-Yau condition on the obstruction bundle. So in order to determine the relative local invariants of all $(X,p_1,\cdots,p_r,q_1,\cdots,q_s)$, we only need to calculate $Z_d(0)(\lambda;x)_{(\mu),(a)}$ because of the gluing law.

Recall that we have defined
$$\xn{K}=\frac{1}{|\Aut(\mu)||\Aut(\gamma)|}\int_{[\Mb]^{\vir}}e(V)$$
with $V$ the obstruction bundle and the generating function
$$\xo{K}(\lambda;x)=\sqrt{-1}^{l(\mu)-d}\sum_{\chi,\gamma}\lambda^{-\chi+l(\mu)}x_\gamma\xn{K}.$$
So we have
$$\xn{K}=Z^{b,\gamma}_d(0)_{(\mu),(a)},$$
where $b=-\chi+n+l(\mu)+\frac{d}{a}-\sum_{i=1}^{n}\frac{\gamma_i}{a}$ and $n=l(\gamma)$. Therefore, we have
\begin{eqnarray*}
\lambda^{-\frac{d}{a}}Z_d(0)(\lambda;\lambda^{\frac{1}{a}-1}x_1,\cdots,\lambda^{\frac{a-1}{a}-1}x_{a-1})
_{(\mu),(a)}&=&\sum_{b,\gamma}Z^{b,\gamma}_d(0)_{(\mu),(a)}\lambda^{-\chi+l(\mu)}x_\gamma\\
&=&\sum_{\chi,\gamma}\xn{K}\lambda^{-\chi+l(\mu)}x_\gamma\\
&=&\sqrt{-1}^{d-l(\mu)}\xo{K}(\lambda;x).
\end{eqnarray*}
Recall that
$$\xo{K}(\lambda;x)=G^\bullet_{\mu}(\lambda;0;x)_{a}=(q^{\frac{1}{2}}q_1^{-\frac{1}{a}}\cdots q_{a-1}^{-\frac{a-1}{a}})^{|\mu|}\sum_{|\nu|=|\mu|}s_{\nu'}(-q_{\bullet})\frac{\chi_{\nu}(\mu)}{z_\mu},$$
where the change of variables is given in Theorem \ref{correspondence}. Therefore, we obtain the expression for $Z_d(0)(\lambda;x)_{(\mu),(a)}$:
\begin{eqnarray*}
\lambda^{-\frac{d}{a}}Z_d(0)(\lambda;\lambda^{\frac{1}{a}-1}x_1,\cdots,\lambda^{\frac{a-1}{a}-1}x_{a-1})
_{(\mu),(a)}=\sqrt{-1}^{d-l(\mu)}(q^{\frac{1}{2}}q_1^{-\frac{1}{a}}\cdots q_{a-1}^{-\frac{a-1}{a}})^{|\mu|}\sum_{|\nu|=|\mu|}s_{\nu'}(-q_{\bullet})\frac{\chi_{\nu}(\mu)}{z_\mu}.
\end{eqnarray*}
This finishes the proof of Theorem 4.

\end{document}